\documentclass[12pt]{amsart}

\usepackage{tikz}
\usepackage{pgflibraryarrows}
\usepackage{pgflibrarysnakes}
\usepackage{amsmath,amsfonts,amssymb}
\usepackage{amsthm}
\usepackage{enumerate}
\usepackage{graphicx}
\usepackage{wrapfig}
\usepackage{mathtools}
\usepackage{epsfig}
\usepackage{caption}
\usepackage{subcaption}
\usepackage{amscd}
\usepackage{url}
\usepackage{enumerate}

\usepackage{tikz}

\DeclareMathOperator{\sinc}{sinc}
\DeclareMathOperator{\dom}{dom \, \mathcal{E}}
\DeclareMathOperator{\domLap}{dom \, \Delta}
\DeclareMathOperator{\R}{\mathbb{R}}
\DeclareMathOperator{\Z}{\mathbb{Z}}

\usepackage{amsthm}
\newtheorem{thm}{Theorem}[section]

\newtheorem{defn}[thm]{Definition}
\newtheorem{lem}[thm]{Lemma}
\newtheorem{cor}[thm]{Corollary}
\newtheorem{con}[thm]{Conjecture}

\author{Robert J. Ravier}
\address{Department of Mathematics \\ Malott Hall \\ Cornell University \\
Ithaca, NY 14853}
\curraddr{Mathematics Department \\ Physics Building \\ Duke University \\ Durham, NC 27708}
\email{robert.ravier@duke.edu}
\thanks{The first author was supported in part by the National Science Foundation, grant DMS-0739164.}

\author{Robert S. Strichartz}
\address{Department of Mathematics \\ Malott Hall \\ Cornell University \\
Ithaca, NY 14853}
\email{str@math.cornell.edu}
\thanks{The second author was supported in part by the National Science Foundation, grant DMS-1162045.}

\subjclass[2010]{28A80}
\keywords{Sierpinski gasket, sampling theory, average values on cells, spectral decimation, bandlimited functions}

\begin{document}

\title[Sampling Theory with Average Values on $SG$]{Sampling Theory with Average Values on the Sierpinski Gasket}

\begin{abstract}
\noindent In the case of some fractals, sampling with average values on cells is more natural than sampling on points.  In this paper we investigate this method of sampling on $SG$ and $SG_{3}$.  In the former, we show that the cell graph approximations have the spectral decimation property and prove an analog of the Shannon sampling theorem. We also investigate the numerical properties of these sampling functions and make conjectures which allow us to look at sampling on infinite blowups of $SG$.  In the case of $SG_{3}$, we show that the cell graphs have the spectral decimation property, but show that it is not useful for proving an analogous sampling theorem. 
\end{abstract}

\maketitle 
\section{Introduction}

Recall that a bandlimited function $f: \mathbb{R} \to \mathbb{R}$ with bandlimit $B$ is a function whose Fourier transform $\hat{f}$ is compactly supported in $[-B,B].$ The classical sampling theorem of Shannon et al. says that such a function $f$ may be explicitly reconstructed from its values $\{ f(k\delta)\vert k \in \mathbb{Z}\}$ for $\delta < \frac{1}{2B}.$ This reconstruction is given in terms of translates of the sinc function $\sinc(x) = \frac{\sin(\pi x)}{\pi x}$.

In this paper, we are interested in analogs of the classical sampling theorem for self-similar subsets of $\mathbb{R}^{2}$; by this, we mean subsets that are equal to a finite union of their images under contractive similarities. Classic examples of such sets include fractals such as the Sierpinski gasket (from here on out denoted $SG$), and an analog of the Sierpinski gasket known as $SG_{3}$. These fractals have well-developed theories for Laplacians (see \cite{kigami2001analysis, MR2246975, drenning2009spectral}) which allow us to do analysis on these sets. We proceed as with the classical case, saying that a function $f$ on either fractal is {\bf{bandlimited}} if its expansion as an infinite series in terms of eigenfunctions $\{ u_{k} \}$ of the fractal's Laplacian $\Delta$ is actually a finite sum where the eigenvalues $\lambda_{k}$ satisfy $\lambda_{k} \le B$, with $B$ being the bandlimit, i.e. the max of the $| \lambda_{k}|$ such that $u_{k}$ appears in the eigenfunction expansion of $f.$ However, in this setting, instead of sampling values of the function at a discrete set of points, we assume instead that we are given the average values of the bandlimited function $f$ over a collection of special subsets, called $m$-cells, that give a natural decomposition of the fractal, where the number of $m$-cells needed for such decomposition increase exponentially as $m$ increases. Geometrically, averages over $m$-cells are more natural samples than a discrete set of points. Sampling with average values in the context of the line is also of great interest, see \cite{olevskii2015local} and the references therein. One may also see \cite{oberlin2003sampling} for a discussion of sampling in the vertex case. The vertex case discussed in \cite{oberlin2003sampling} and the cell case discussed in this paper are very different; the vertex case considers eigenfunctions corresponding to the Dirichlet spectrum of the Laplacian, whereas the cell case discussed in this paper will consider eigenfunctions corresponding to the Neumann spectrum. Also, the appropriate bandlimits in the two cases are different.

An important property of $SG$ and $SG_{3}$ is that the spectra of their respective natural Laplacians are describable explicitly by the method of spectral decimation introduced by Fukushima and Shima \cite{fukushima1992spectral}, which relates eigenfunctions and eigenvalues of the Laplacians on $SG$ and $SG_{3}$ with eigenfunctions and eigenvalues of the discrete Laplacians of their graph approximations. The authors of \cite{oberlin2003sampling} showed how to use the spectral decimation method to obtain a sampling theorem involving discrete point samples on the vertices of the graph approximations of $SG.$ On the other hand, it was shown in \cite{MR1867899} that the Laplacian on $SG$ is also definable as a limit of graph Laplacians on the graphs whose vertices correspond to the $m$-cells of $SG$ with edges between vertices if two nonidentical $m$-cells share a vertex in common. The main technical result of this paper is to show that the method of spectral decimation is also valid for the sequence of cell graph approximations. This leads directly to a sampling theorem for average value samples.

In contrast, it turns out that despite $SG_{3}$ having graph approximations with an explicit spectral decimation process (see \cite{drenning2009spectral}), as well as approximations of its Laplacian in terms of discrete Laplacian on cell graph approximations akin to those for $SG$ (see \cite{donglei2005laplacian}), there is no useful spectral decimation for the standard cell graph Laplacian. In fact, we will show that there is no such Laplacian for the cell graph approximations that results in a sampling theorem like that for $SG.$

The paper is organized as follows. In section 2 we summarize the necessary theory of the Kigami Laplacian on $SG$ needed to prove our sampling theorem for $SG.$ In section 3 we describe the cell graph approximations $\Gamma_{m}$ of $SG,$ their spectral decimation results, and their eigenbases. In section 4 we prove the sampling theorem for $SG$. In section 5 we present numerical data on the cardinal interpolants, and give some conjectures concerning possible exponential localization. This is quite a contrast to the poor localization of the sinc function. We show how the conjectures imply a sampling theorem on infinite blowups of $SG$. In section 6 we present the negative results for $SG_{3}$ .

\section{$SG$ Preliminaries}
In this section, we summarize all relevant results from prior work. Unless otherwise noted, full details can be found in \cite{MR2246975}.
\subsection{Self-Similar Structures, the Construction of $SG,$ and Cells}

We begin with a definition.

\begin{defn}
We say that a connected, compact set of $\mathbb{R}^{d}$ $K$ has a {\bf{self-similar structure}} if there exist a finite number of homeomorphisms $F_{1},...,F_{n}$ with $F_{i}: K \to W_{i}$ such that $W_{i} \subset K$ and $\bigcup_{i=1}^{n} W_{i} = K.$
\end{defn}

\noindent Observe that the unit cube in $\mathbb{R}^{d}$ has a self-similar structure. To construct a more topologically interesting subset of $\mathbb{R}^{2},$ we proceed as follows: let $q_0, q_1, q_2$ denote the vertices of an equilateral triangle in $\mathbb{R}^{2}.$ For simplicity, we let $q_{0} = (0,0),$ $q_{1} = \left(\frac{1}{2}, \frac{\sqrt{3}}{2} \right),$ and $q_{2} = (1,0).$ Define $F_{i}: \mathbb{R}^2 \rightarrow \mathbb{R}^2$ by

$$F_{i}(x) = \frac{1}{2}(x-q_{i}) + q_{i}. \eqno(2.1)$$

\noindent for $i = 0, 1, 2.$ 

\begin{defn}
The {\bf{Sierpinski gasket}} $SG$ is the unique nonempty compact set satisfying 
$$SG = \bigcup_{i =1}^3 F_{i}(SG).$$
\end{defn}

To construct (and do computations on) $SG,$ we use graph approximations. First, some notation. We define a {\bf{word of length $n$}}, $(w_{1}, ..., w_{n})$ to simply be an element of $\mathbb{Z}_{3}^{n}$. Then, we say that $F_{w} = F_{w_{n}} \circ ... \circ F_{w_{1}}$. Let $\beta_{0}$ be the graph with vertices $V_{0} = \{ q_{0}, q_{1}, q_{2} \}$ and edges between $q_{i}$ and $q_{j}$ for $i, j = 0,1,2, i \neq j.$ In other words, $\beta_{0}$ is the graph representing the equilateral triangle $T$ determined by $q_{0}, q_{1},$ and $q_{2}.$ We then inductively define the graphs $\beta_{m}.$ For $m > 0$ an integer, we let

$$V_{m} = \bigcup_{i=0}^{2} F_{i}(V_{m-1}).$$

\noindent Then, we define $\beta_{m}$ to be the graph with vertices $V_{m}$ such that $x, y \in V_{m}$ are connected by exactly one edge if and only if $x \neq y$ and $x,y \in F_{w}(T)$ for some word $w$ of length $m,$ i.e. $x$ and $y$ are in the same $m$-cell. Examples are shown in Figure 1. With this in mind, we let $V_{\ast} = \lim_{m \to \infty} V_{m}$ be the vertices of the limiting graph $\beta_{\ast}$ of the $\beta_{m}.$ Then $SG$ is the completion of $V_{\ast}$ in $\R^{2}.$ 

We now define some important terminology. We say that the set $K$ is an {\bf{$m$-cell of $SG$}} if $K = F_{w}(SG)$ for some word $w \in \mathbb{Z}_{3}^{m}.$ Also, for $n \geq m,$ we say that the graph $W$ is an $m$-cell of $\beta_{n}$ if $W = F_{v}(\beta_{n-m})$ for some word $v \in \mathbb{Z},$ We will frequently deal with $m$-cells on $\beta_{m}$ in our computations.

Before continuing on to the next section, we make two important observations: first, for $m \geq 1$ and $v \in V_{m} \backslash V_{0},$ $v$ belongs to two different $m$-cells of $SG.$ In particular, for every $v \in V_{m} \backslash V_{0},$ we can find exactly two distinct $i, j \in \Z_{3}$ and exactly two distinct words $w$ and $w'$ of length at most $m$ such that $F_{w}(q_{i}) = v = F_{w'}(q_{j}).$ On the other hand, we observe that for every $m$-cell $A$, there is a unique word of length $m$ $w$ such that $A = F_{w}T.$

\begin{figure}
    \centering
    \begin{subfigure}[b]{0.3\textwidth}
        \centering
        \includegraphics[width=\textwidth]{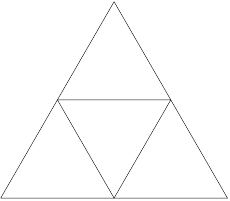}
        
    \end{subfigure}%
    ~ 
    \begin{subfigure}[b]{0.3\textwidth}
        \centering
        \includegraphics[width=\textwidth]{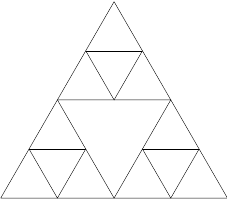}
        
    \end{subfigure}
        \caption{$\beta_{1}$ and $\beta_{2}.$}
\end{figure}

\subsection{Integration and Averaging on $SG$ and $\beta_{m}$}
For the purposes of this paper, any integration on $SG$ is done with respect to the measure that assigns a measure of 1 to $SG$ and $3^{-m}$ to each $m$-cell, so $SG$ with this measure is a probability space.

We define the integral of a continuous function $f: SG \to \R$ in terms of Riemann sums, i.e.

$$\int_{SG} f \, d \mu = \lim_{m \to \infty} \sum_{|w| = m} f(x_{w}) \mu(F_{w}(SG))$$

\noindent where each $x_{w}$ is some point in $F_{w}(SG).$ An argument yields the more computationally effective formula:

$$\int_{SG} f \, d \mu = \lim_{m \to \infty} \frac{1}{3} \sum_{i=0}^{2} f(F_{w}q_{i}) \mu(F_{w}(SG)) \eqno(2.2). $$

\noindent We now define discrete measures $\nu_{m}$ on $\beta_{m}$ by the following rule: 

$$\text{If } v \in V_{0}, \text{ then } \nu_{m}(v) = \frac{1}{3^{m+1}}, \text{ and if }w \in V_{m} \backslash V_{0}, \text{ then } \nu_{m}(w) = \frac{2}{3^{m+1}}.$$

\noindent Using this along with (1), we see that

$$\int_{\beta_{m}} f \, d \nu_{m} = \frac{1}{3} \sum_{i=0}^{2} f(F_{w}q_{i}) \mu(F_{w}(SG)),$$

\noindent so

$$\int_{SG} f \, d \mu = \lim_{m \to \infty}\int_{\beta_{m}} f \, d \nu_{m}.$$

\noindent Integration over a subset $A$ of $SG$ is defined analogously, making use of the characteristic function $\chi_{A}$ when appropriate. We make an important remark: by construction of the Sierpinski gasket, we have for an $m$-cell $K,$ we have, for $\theta = \mu$ or $\nu_{n}$ for $n > m,$

$$\int_{K} f \, d \theta = \sum_{i=0}^{2} \int_{F_{i}K} f \, d \theta. \eqno(2.3)$$

We now define two averages. The discrete average on an $m$-cell $C$ corresponding to the word of length $m,$ $w,$ for a function on either $SG$ or $\beta_{n}$ for $n$ at least $m$ is given by:

$$A_{C}(u) = \frac{1}{3} \sum_{i=0}^{2} u(F_{w} q_{i}) \eqno(2.4)$$

\noindent where $\partial C$ refers to the boundary points of $C$. The continuous average is given by

$$B_{C}(u) = 3^{m} \int_{C} u \, d \mu \eqno(2.5).$$

\subsection{Energy and Distance}
For real-valued functions $u,v$ on $\beta_{m},$ we define the graph energy pairing

$$E_{m}(u,v) = \sum_{x \sim y} (u(x)-u(y))(v(x)-v(y))$$

\noindent where $x \sim y$ means there is an edge between $x$ and $y,$ so the sum ranges over the edges of the $\beta_{m}.$ It's easy to see that $E_{m}(u,u) = 0$ if and only if $u$ is a constant function. We then define

$$\mathcal{E}_{m}(u,v) = \left( \frac{5}{3} \right)^{m} E_{m}(u,v).$$

\noindent Then, for two real valued function $u$ and $v$ on $SG,$ we define the graph energy bilinear pairing

$$\mathcal{E}(u,v) = \lim_{m \to \infty} \mathcal{E}_{m}(u,v)$$

\noindent if the limit exists. Since $\mathcal{E}_{m}(u,u)$ is nondecreasing and nonnegative, this limit always exists in the case $u=v.$

\begin{defn}
We say that a real-valued function $u$ on $SG$ has {\bf{finite energy}} if

$$\mathcal{E}(u,u) < \infty.$$

\noindent The set of such functions is denoted $\dom$.
\end{defn}

We can use this to define a metric on $SG.$

\begin{defn}
The {\bf{resistance metric}} on $SG$ is the function $R: SG \times SG \to \R$ is the function such that 

$$R(x,y) = \inf \{C \in \R: |u(x)-u(y)|^{2} \le C \mathcal{E}(u,u) \, \forall u \in \dom \}. \eqno(2.6)$$

\noindent An important property of $R(x,y)$ is that $c|x-y|^{\beta} \le R(x,y) \le C|x-y|^{\beta}$ for some constants $c$ and $C,$ where $\beta = \log(5/3) / \log(2).$

\end{defn}

We also note the following fact that will be useful in the subsequent section

\begin{thm}
The space $\dom$ modulo the constant functions is a Hilbert space with inner product $\mathcal{E}(u,v).$
\end{thm}
\subsection{The Laplacian}

We define the weak Laplacian on $SG$ by the following:

\begin{defn}
Let $u \in \dom$. Then $\Delta u = f$ for $f$ continuous if

$$\mathcal{E}(u,v) = - \int \limits_{SG} fvd\mu \eqno(2.7)$$

\noindent for all $v \in \dom$. If $\Delta u$ exists, we say that $u \in \domLap.$
\end{defn}

It turns out that the weak Laplacian on $SG$ has a pointwise formula. We define the graph Laplacian of a function $u$ on $\beta_{m}$ at vertex $x \in V_{m} \backslash V_{0}$ to be

$$\Delta_{\beta_{m}}u(x) = \sum_{x \sim y} (u(y)-u(x)). \eqno(2.8)$$

\noindent We define the renormalized graph Laplacian $\Delta_{m} = \frac{3}{2}5^{m} \Delta_{\beta_{m}}.$ We do not define the Laplacian on $q_{0}, q_{1}, q_{2},$ keeping with our defining them as boundary points. It turns out that the $\Delta_{m}$ well-approximates the Laplacian on $SG$ in the followin sense.

\begin{thm}
If $u \in \domLap,$ then we have

$$\Delta u(x) = \lim_{m \to \infty} \Delta_{m}u(x) \eqno(\ast)$$

\noindent uniformly for all $x \in V_{\ast} \backslash V_{0}.$ Conversely, if $u$ is a continuous function on $V_{\ast} \backslash V_{0}$ and the right side of ($\ast$) converges uniformly to a continuous function on $V_{\ast} \backslash V_{0},$ then the extension of $u$ to $SG$ is in $\domLap$ and ($\ast$) holds for all $x \in SG.$
\end{thm}

For a proof, see (\cite{kigami2001analysis, MR2246975}). In the following, we deal with the graph Laplacian $\Delta_{\beta_{m}}.$

\subsection{Spectral Decimation}
A key feature of the $\beta_{m}$ graphs is that every eigenfunction on $\beta_{k}$ naturally continues to an eigenfunction on $\beta_{k+1},$ and hence every eigenfunction on $\beta_{k}$ naturally continues to an eigenfunction on $\beta_{n}$ for $n > k.$ We make this precise:

\begin{defn}
A sequence of graphs $G_{m}, m \in \mathbb{N},$ each with Laplacians $\Delta_{m}$ and vertices $V_{m}$ such that the cardinality of $V_{m}$ is strictly increasing and tends to infinity satisfies the {\bf{spectral decimation property}} if there are onto linear operators $R_{m}$ from the space of real valued functions on $G_{m+1}$ to the space of real valued functions on $G_{m}$ such that

(1) There is a finite number of eigenvalues $\alpha_{1},...\alpha_{n}$ (one of which is 0) independent of $m$ such that the kernel of $R_{m}$ consists of the subspace of real valued functions on $G_{m+1}$ spanned by the eigenfunctions on $G_{m+1}$ with the $\alpha_{i}$ as their eigenvalues.

(2) If $u_{m+1}$ is an eigenfunction on $G_{m+1}$ with eigenvalue $\lambda_{m+1}$ that does not lie in $\ker R_{m},$ then $R_{m}(u_{m+1})$ is a non-zero eigenfunction on $G_{m}$ with eigenvalue $\lambda_{m}$ with finite inverse image.

(3) Keeping the notation of (2), there exists a rational function $f: \mathbb{R} \to \mathbb{R}$ independent of $\lambda_{m+1}$ such that $\lambda_{m} = f(\lambda_{m+1}).$
\end{defn}

\noindent The eigenvalues $\alpha_{1},..., \alpha_{n}$ in Definition 2.8 are known as {\bf{forbidden eigenvalues}}. To show that a sequence of graphs $G_{m}$ satisfies Definition 2.8, we will start with an eigenfunction $f$ on $G_{m}$ and a formula to "continue" $f$ to an eigenfunction $f'$ on $G_{m+1}.$ From this situation, the operators $R_{m}$ will be obvious. Note that in this definition, we do not require that $G_{m}$ is a subgraph of any $G_{n}$ for $m < n,$ so we mean continuation in a rather loose sense.

It turns out that the $\beta_{m}$ satisfy the spectral decimation property. To see this, consider Figure 2. 

\begin{figure}
    \centering
\includegraphics[width =.3\textwidth]{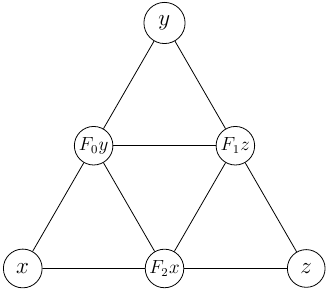}
\caption{A general $(m-1)$-cell of $\beta_{m}.$}
\end{figure}

The $(m-1)$-cell has boundary $x, y, z.$ Suppose that $u$ is an eigenfunction on $\beta_{m-1}$ with eigenvalue $\lambda_{m-1}.$ We would like to continue $u$ to an eigenfunction $u'$ on $\beta_{m}$ with eigenvalue $\lambda_{m}$ such that $u(w) = u'(w)$ for all $w \in V_{m-1}.$ It turns out that if we choose $\lambda_{m}$ such that it satisfies $\lambda_{m-1} = \lambda_{m}(5-\lambda_{m}),$ i.e.

$$\lambda_{m} = \frac{5 + \varepsilon_{m} \sqrt{25-4 \lambda_{m-1}}}{2} \eqno(2.9)$$

\noindent where $\varepsilon_{m} = \pm 1$ and define $u'$ by letting $u' = u$ on $\beta_{m-1}$ and

$$u'(F_{0}y)= \frac{4-\lambda_{m})(u(x)+u(y))+2u(z)}{(2-\lambda_{m})(5-\lambda_{m})} \eqno(2.10)$$

$$u'(F_{1}z) = \frac{4-\lambda_{m})(u(y)+u(z))+2u(x)}{(2-\lambda_{m})(5-\lambda_{m})} \eqno(2.11)$$

$$u'(F_{2}x) = \frac{4-\lambda_{m})(u(x)+u(z))+2u(y)}{(2-\lambda_{m})(5-\lambda_{m})} \eqno(2.12)$$

\noindent then $u'$ is an eigenfunction on $\beta_{m}$ with $\lambda_{m}$ as its eigenvalue. More precisely:

\begin{thm}
Suppose $\lambda_{m} \neq 2, 5,$ or 6, and $\lambda_{m-1}$ is given by

$$\lambda_{m-1} = \lambda_{m}(5-\lambda_{m}).$$

\noindent If $u$ is an eigenfunction on $\beta_{m-1}$ with eigenvalue $\lambda_{m-1}$ that is continued to a function $u'$ on $\beta_{m}$ by (2.10)-(2.12), then $u'$ is an eigenfunction on $\beta_{m}$ with eigenvalue $\lambda_{m}.$ Conversely, if $u'$ is an eigenfunction on $\beta_{m}$ with eigenvalue $\lambda_{m},$ then the restriction of $u'$ to $\beta_{m-1}$ is an eigenfunction on $\beta_{m-1}$ with eigenvalue $\lambda_{m-1}.$
\end{thm}

\noindent In this case, the numbers 2, 5, and 6 are the forbidden eigenvalues. That 2 and 5 are not allowed follows from the formulas (2.10)-(2.12). That 6 is not allowed follows from the proof that $\lambda_{m}$ must satisfy (2.9); see \cite{MR2246975} for the full details.

We note the following: if we have a sequence of eigenvalues $\{ \lambda_{m} \}_{m=1}^{\infty},$ where $\lambda_{n}$ and $\lambda_{n+1}$ are related by (2.9), we only allow a finite number of the $\varepsilon_{m}$ to equal 1.

\subsection{The Neumann Eigenbasis of $\beta_{m}$}
In this paper, we concern ourselves with the Neumann eigenfunctions of $\beta_{m},$ which are eigenfunctions $u$ of $\Delta_{\beta_{m}}$ with eigenvalue $\lambda_{m}$ that satisfy the following conditions:

$$(4-\lambda_{m})u(q_{0}) = 2(u(F_{0}^{m}q_{1})+u(F_{0}^{m}q_{2})) \eqno(2.13)$$
$$(4-\lambda_{m})u(q_{1}) = 2(u(F_{0}^{m}q_{0})+u(F_{0}^{m}q_{2})) \eqno(2.14)$$
$$(4-\lambda_{m})u(q_{2}) = 2(u(F_{0}^{m}q_{0})+u(F_{0}^{m}q_{1})) \eqno(2.15)$$

\noindent It is known that the Neumann eigenfunctions on $\beta_{m}$ form a vector space of dimension $\frac{3^{m+1}+3}{2},$ which is the number of vertices of $\beta_{m}.$

In this section, we list all of the Neumann eigenfunctions on $\beta_{m},$ as we'll need this list for the proof of our sampling theorem in Section 4. For complete details, see Section 3.3 of \cite{MR2246975}.

In Figure 3, we list the Neumann eigenbasis of $\beta_{0}.$ The basis consists of a constant eigenfunction with eigenvalue 0 and two rotations of a nonconstant eigenfunction with eigenvalue 6 (recall that we have not defined the graph Laplacian on $\beta_{0}$ and are instead forcing the Neumann conditions (2.13)-(2.15) on each vertex).

\begin{figure}
		\centering
    \begin{subfigure}[b]{0.3\textwidth}
        \centering
        \includegraphics[scale = .6]{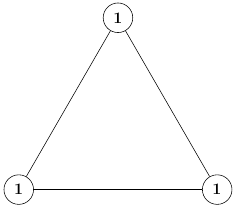}
    \end{subfigure}%
    ~ 
    \begin{subfigure}[b]{0.3\textwidth}
        \centering
        \includegraphics[scale = .6]{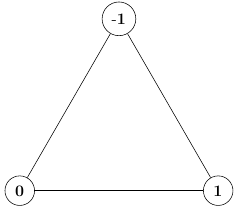}
    \end{subfigure}
    
    \caption{The Neumann eigenbasis of $\beta_{0}$ consists of the constant eigenfunction on the left and two rotations of the eigenfunction on the right. The left eigenfunction has eigenvalue 0, whereas the one on the right has eigenvalue 6.}
\end{figure}

The basis for $\beta_{1}$ consists of all three rotations of the eigenfunction with eigenvalue 6 in Figure 4, as well as the continuation of the eigenfunctions on $\beta_{0}$ above to eigenfunctions on $\beta_{1}$ via spectral decimation. The constant function is continued to the constant function, and the two eigenfunctions on the Neumann eigenbasis for $\beta_{0}$ are continued to eigenfunctions (We make the remark that while equation (2.6) gives two possible eigenvalues for a continuation, Theorem 2.2 asserts that we cannot continue an eigenfunction to the eigenvalues 2, 5, and 6. Thus, in the case of the Neumann eigenbasis of $\beta_{0},$ each function has precisely one continuation).

\begin{figure}
\centering
\includegraphics[width =.3\textwidth]{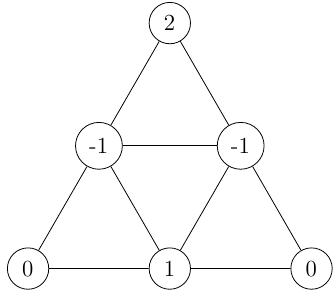}
\caption{An element of the Neumann eigenbasis for $\beta_{1}$ with eigenvalue 6. The other two elements of the eigenbasis with eigenvalue 6 are the rotations of this one.}
\end{figure}

\begin{figure}
\centering
\includegraphics[width =.3\textwidth]{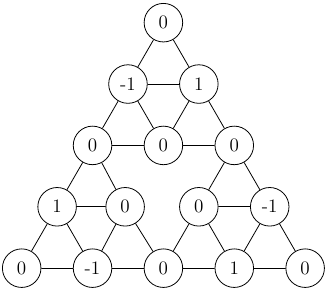}
\caption{The element of the Neumann eigenbasis for $\beta_{2}$ with eigenvalue 5.}
\end{figure}

We get eight linearly independent Neumann eigenfunctions on $\beta_{2}$ by continuing the six on $\beta_{1}$ by spectral decimation. We get one eigenfunction of multiplicity five by alternating placing 1 and -1 around the cycle corresponding to the downward pointing triangle of side length $\frac{1}{2}.$ See Figure 5 for details. We get a total of six eigenfunctions with eigenvalue 6 by considering the three rotations of the two functions in Figure 6.

\begin{figure}
		\centering
    \begin{subfigure}[b]{0.4\textwidth}
        \centering
        \includegraphics[scale = .55]{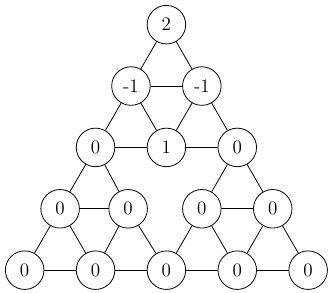}
    \end{subfigure}%
    ~ 
    \begin{subfigure}[b]{0.4\textwidth}
        \centering
        \includegraphics[scale = .55]{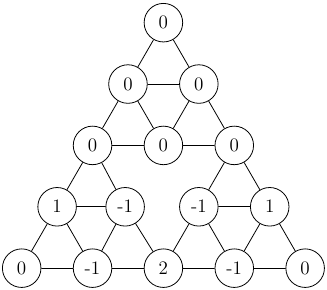}
    \end{subfigure}
    
    \caption{Two elements of the Neumann eigenbasis for $\beta_{2}$ with eigenvalue 6. The remaining elements of the Neumann eigenbasis for $\beta_{2}$ with eigenvalue 6 are the rotations of these two functions.}
\end{figure}

For the general case, we get Neumann eigenfunctions with eigenvalue 5 by alternating placing 1 and -1 around each of the downward pointing triangle 1-cycle of side length greater than $\frac{1}{2^{m}}.$ By induction, there are $\frac{3^{m-1}-1}{2}$ such triangles, hence there are $\frac{3^{m-1}-1}{2}$ such Neumann eigenfunctions with eigenvalue 5. To get Neumann eigenfunctions with eigenvalue 6, we can essentially copy the eigenfunctions listed in Figure 6. Specifically, we can place a 2 at each element of $V_{0}$ and repeat the construction of the eigenfunction in Figure 7(a), making appropriate adjustments, to get three functions with eigenvalue 3. We can then place a 2 at each element of $V_{m-1} \backslash V_{0}$ and then repeat the construction (with appropriate adjustments) of the eigenfunction in Figure 7(b) to get $\frac{3^{m}-3}{2}$ additional Neumann eigenfunctions with eigenvalue 6. Combining these eigenfunctions with the ones obtained from continuation of the Neumann eigenbasis of $\beta_{m-1},$ we obtain the Neumann eigenbasis for $\beta_{m}.$

We now make two remarks. First we say that an element of the Neumann eigenbasis of $\beta_{m}$ is {\bf{generated at level $m$}}, or alternatively has {\bf{generation of birth at level $m$}} if it is not the continuation via spectral decimation of any Neumann eigenbasis element on $\beta_{m-1}.$ Second, we note that it is a fact that the largest eigenvalue on any $\beta_{m}$ is 6, and the second largest is 5. This will come into play in an important definition later.

We now consider Neumann eigenfunctions on $SG$ with the Laplacian $\Delta.$ An eigenfunction $u$ on $SG$ is a Neumann eigenfunction if, for $i \in \Z_{3},$ we have that the normal derivative

$$\partial_{n} u(q_{i}) = \lim_{m \to \infty} \left( \frac{5}{3} \right)^{m} (2u(q_{i}) + u(F_{i}^{m} q_{i+1}) - u(F_{i}^{m} q_{i-1})) = 0.$$

\noindent It turns out that every Neumann eigenfunction $u$ on $SG$ with eigenvalue $\lambda$ can be obtained from applying the spectral decimation formulas (2.10)-(2.12) to some Neumann eigenfunction $u_{m_{0}}$ on $\beta_{m_{0}}.$ Specifically, for every Neumann eigenfunction $u$ on $SG$ with eigenvalue $\lambda,$ there is some Neumann eigenfunction $u_{m_{0}}$ with generation of birth $m_{0}$ and a sequence of extensions $u_{m}$ via spectral decimation with eigenvalues $\lambda_{m}$ such that $u_{m} \to u$ uniformly and

$$\lambda = \lim_{m \to \infty} \frac{3}{2}5^{m} \lambda_{m}. \eqno(2.16)$$

\noindent As with the Laplacian, we ignore the renormalization constant $\frac{3}{2}5^{m}$ for all computations on the $\beta_{m}.$

\subsection{The Neumann Spectrum of $SG$}

We end the overview of analysis on $SG$ by discussing briefly the Neumann spectrum of $SG.$ Specifically, we wish to somehow classify the smallest Neumann eigenvalues. To do this, we recall from the previous subsection that Neumann eigenfunctions on $SG$ are obtained by continuing Neumann eigenfunctions from the $\beta_{m}$ graphs via spectral decimation. We observe that for all $m,$ $\lambda_{m+1} < \lambda_{m}$ if $\varepsilon_{m+1} = -1$ by (2.9). We also see that if $\lambda_{m} = 6,$ $\varepsilon_{m+1}$ can only be 1 as $\varepsilon_{m+1} = -1$ yields $\lambda_{m+1} = 2,$ which is a forbidden eigenvalue. With this information, as well as the discussion in the previous subsection, one can see that we get the smallest $\frac{3^{m+1}+3}{2} - \frac{3^{m}+3}{2} = 3^{m}$ eigenvalues of the Neumann spectrum of $SG$ by taking the Neumann eigenvalues of $\beta_{m}$ that are not 6, continuing them via (2.9) by letting $\varepsilon_{n} = -1$ for all $n > m,$ and using the limit definition (2.13). The continuations of the Neumann eigenfunctions on $\beta_{m}$ corresponding to these eigenvalues on $\beta_{m}$ and their aforementioned continuations to Neumann eigenvalues on $SG$ will be of particular interest in the sampling theorem. We call these functions {\bf{the first $3^{m}$ elements of the Neumann eigenbasis of $SG$}}. One can show via spectral decimation that these first $3^{m}$ Neumann eigenfunctions have eigenvalues bounded by $B5^{m}$ for some $B.$
\section{The Average Cell Graphs $\Gamma_{m}$}

\subsection{Construction of the Graphs}

In this section, we introduce a new family of graphs approximating $SG.$ For $m$ a nonnegative integer, we obtain $\Gamma_{m}$ from $\beta_{m}$ as follows: every vertex of $\Gamma_{m}$ corresponds to an $m$-cell on $\beta_{m},$ and two vertices are connected by an edge if and only if the corresponding $m$-cells share exactly one boundary point (namely, no vertex on $\Gamma_{m}$ is connected to itself). For example, since $\beta_{0}$ is just the a 0-cell, $\Gamma_{0}$ is the graph with one vertex and no edges. Both $\Gamma_{1}$ and $\Gamma_{2},$ as well as the graphs $\beta_{1}$ and $\beta_{2}$ that they are derived from, are listed in Figures 7 and 8. An important thing to note is that while $\beta_{m}$ naturally contained the vertices of $\beta_{m-1}$ for $m \geq 1,$ there is no such preservation for the $\Gamma_{m}$ graphs. In a sense, when moving from $\Gamma_{m}$ to $\Gamma_{m+1},$ every vertex in $\Gamma_{m}$ splits into three distinct vertices of $\Gamma_{m+1}.$ An example of this splitting can be seen by looking at $\Gamma_{1}$ and $\Gamma_{2}$ in Figures 8 and 9. We define the boundary points of $\Gamma_{m}$ to be the vertices with degree 2 (all other vertices have degree 3); note that the boundary points are not preserved from $\Gamma_{m}$ to $\Gamma_{m+1}$ since no vertex is preserved. 

\begin{figure}
		\centering
    \begin{subfigure}[b]{0.3\textwidth}
        \centering
        \includegraphics[width=\textwidth]{Beta1.png}
    \end{subfigure}%
    ~ 
    \begin{subfigure}[b]{0.3\textwidth}
        \centering
        \includegraphics[width=\textwidth]{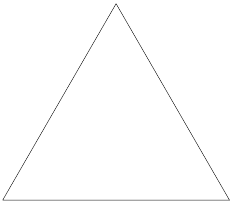}
    \end{subfigure}
    
    \caption{The graphs $\beta_{1}$ and $\Gamma_{1}.$}
\end{figure}

\begin{figure}
		\centering
    \begin{subfigure}[b]{0.3\textwidth}
        \centering
        \includegraphics[width=\textwidth]{Beta2.png}
    \end{subfigure}%
    ~ 
    \begin{subfigure}[b]{0.3\textwidth}
        \centering
        \includegraphics[width=\textwidth]{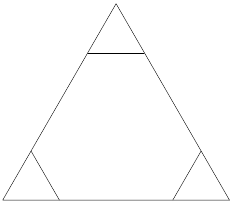}
    \end{subfigure}
    
    \caption{The graphs $\beta_{2}$ and $\Gamma_{2}.$}
\end{figure}

We also recall that since $m$-cells are uniquely indexed by words of length $m,$ so are the vertices of $\Gamma_{m}.$

Given any function $f$ on $\Gamma_{m},$ we want its value on a vertex $x$ to equal the average value of a related function $g$ on $\beta_{m}$ on the corresponding $m$-cell, either in the sense of (2.4) or (2.5) (we will use both). If $u$ is a function on $\Gamma_{m}$ and $x$ is a vertex of $\Gamma_{m}$ corresponding to the $m$-cell on $\beta_{m}$ with vertices $a_{x}, b_{x},$ and $c_{x},$ then there is some function $v$ on $\beta_{m}$ such that for every vertex $x$ in $\Gamma_{m},$ we have

$$u(x) = \frac{v(a_{x})+v(b_{x})+v(c_{x})}{3}. \eqno(3.1)$$ 

\noindent Since we want functions on $\Gamma_{m}$ to correspond to average values of functions on $m$-cells (see (2.4)-(2.5)), we would like for these functions to satisfy property (2.3). We define continuation of a function on $\Gamma_{m}$ to one on $\Gamma_{m+1}$ in the following sense. Assume that a vertex $x$ on $\Gamma_{m}$ splits into three distinct vertices $w, y, $and $z$ on $\Gamma_{m+1}.$ Then we say that a function $g$ on $\Gamma_{m+1}$ is a {\bf{continuation of a function $f$ on $\Gamma_{m}$}} provided that 

$$f(x) = \frac{g(w)+g(y)+g(z)}{3} \eqno (3.2)$$.

We can define integration over $\Gamma_{m}$ by assigning the measure $\frac{1}{3^{m}}$ to each vertex, analogous to our assigning the measure of $\frac{1}{3^{m}}$ to each $m$-cell in $SG.$ Let $\rho_{m}$ denote this measure on $\Gamma_{m}.$ If $W_{m}$ is the set of vertices on $\Gamma_{m},$ we then have

$$\int_{\Gamma_{m}} f \, d \rho_{m} = \frac{1}{3^{m}}\sum_{x \in W_{m}} f(x) \eqno(3.3)$$

\noindent for any function $f$ on $\Gamma_{m}.$ If $f(x)$ is the function corresponding to the average values of some function $g$ on $\beta_{m},$ then (3.1)-(3.3) and a little bit of algebra gives

$$\int_{\Gamma_{m}} f \, d \rho_{m} = \int_{\beta_{m}} g \, d \nu_{m}.$$

We can define the Laplacian on $\Gamma_{m}$ in the same manner that we define it on $\beta_{m}.$ However, there is one caveat: we define the Laplacian at {\emph{every}} vertex in $\Gamma_{m}.$ This keeps with the fact that there really is no good definition of a boundary vertex for $\Gamma_{m}$ in contrast to the $\beta_{m}$ case. Specifically, for $x$ a vertex in $\Gamma_{m},$ we define the graph Laplacian

$$\Delta_{\Gamma_{m}}u(x) = \sum_{x \sim y} (u(y)-u(x)) \eqno(3.4) $$

\noindent and we defined the renormalized graph Laplacian $\Delta_{m}=\frac{3}{2}5^{m} \Delta_{\Gamma_{m}}.$ Note that we also use the symbol $\Delta_{m}$ for the renormalzied graph Laplacian on $\beta_{m}.$ It will be clear from the context which we will use. 

It turns out that we have an analog of Theorem 2.6 in the case of $\Delta_{\Gamma_{m}}.$ Precisely, letting $\Delta_{m}$ denote the renormalized graph Laplacian for $\Gamma_{m},$ we have

\begin{thm}
Let $\Delta u = f.$ Then $\Delta_{m} u$ converges uniformly to $f$. Conversely, if $u$ is integrable, and $\Delta_{m}u$ converges uniformly to a continuous function $f$, then $\Delta u$ exists and equals $f$.
\end{thm}

\noindent For a proof of this, see \cite{MR1867899}. In the following, we work with the graph Laplacian $\Delta_{\Gamma_{m}}.$

\subsection{Spectral Decimation on $\Gamma_{m}$}

It turns out that the $\Gamma_{m}$ graphs satisfy the spectral decimation property listed in Defintiion 2.7. 

\begin{thm}[Spectral Decimation]
Let $u$ be an eigenfunction on $\Gamma_{m}$ with eigenvalue $\lambda_{m}$. Then, $u$ can be continued to at most two eigenfunctions on $\Gamma_{m+1}$ with eigenvalues $\lambda^{(1)}_{m+1}$ and $\lambda^{(2)}_{m+1}$. Furthermore, for each $\lambda^{(k)}_{m+1}$, the corresponding continuation is unique, and

$$u(x) = \frac{1}{3} \sum_{i=0}^{2} u'(F_{i}x) \eqno(\ast \ast)$$

\noindent holds

Conversely, if $u'$ is an eigenfunction on $\Gamma_{m+1}$ with eigenvalue $\lambda^{(1)}_{m+1}$ or $\lambda^{(2)}_{m+1}$, then $u$, the function on $\Gamma_{m}$ defined by $( \ast \ast)$ is an eigenfunction on $\Gamma_{m}$ with eigenvalue $\lambda_{m}$. The relationship between $\lambda_{m}$ and $\lambda^{(k)}_{m+1}$ is given by

$$\lambda^{(k)}_{m+1} = \frac{5 + \varepsilon_{m+1} \sqrt{25-4 \lambda_{m}}}{2}$$

\noindent where $\varepsilon_{m+1} = \pm 1.$
\end{thm}

Before beginning the proof, we remark that these continuations are not pure extensions since no vertices are preserved when moving from $\Gamma_{m}$ to $\Gamma_{m+1}.$

\begin{proof}
The case $m=0$ is trivial. Consider Figure 9. The picture on the left details a general subgraph of the interior of $\Gamma_{m}$ for $m>1$ (The case for $m = 1$ can be obtained by simply deleting the point labeled W and its corresponding edge connecting it to the triangular group of vertices). By symmetry we can assume that this is indeed the general subgraph. We continue it to a corresponding general subgraph of $\Gamma_{m+1}$ in the picture adjacent to it. Since $u$ is an eigenfunction on $\Gamma_{m},$ we have

$$(3-\lambda_{m})u(X) = u(W) + u(Y) + u(Z).\eqno(3.5)$$

Now, assume that $u$ continues to an eigenfunction $u'$ on $\Gamma_{m+1}$. This means that, for vertex $F_{1}X$,

$$(3-\lambda_{m+1})u'(F_{1}X) = u'(F_{0}X) + u'(F_{2}X)+u'(F_{0}Y)\eqno(3.6)$$

\noindent and similarly for every vertex except possibly for $v_{5}$, $v_{9}$, and $v_{a}$ ($v_{1}$ in the case $m=1$), as one of them might be a boundary vertex. By adding $u'(F_{1}X)$ to both sides of (3.6) and using the mean value property (3.2) or ($\ast \ast)$ of the cell graph, we obtain

$$(4-\lambda_{m+1})u'(F_{1}X) = u'(F_{0}Y) + 3u(X) \eqno(3.7).$$

\noindent By a similar argument, we get

$$(4-\lambda_{m+1})u'(F_{0}Y) = u'(F_{1}X) + 3u(Y) \eqno(3.8).$$

\noindent Solving (3.8) for $u'(F_{0}Y)$ and substituting the resulting expression in (3.7), we can then solve for $u'(F_{1}X)$ to obtain

$$u'(F_{1}X) = \frac{3(4-\lambda_{m+1})u(X)+3u(W)}{(3-\lambda_{m+1})(5-\lambda_{m+1})}\eqno(3.8)$$

\begin{figure}
		\centering
    \begin{subfigure}[b]{0.3\textwidth}
        \centering
        \includegraphics[scale = .4]{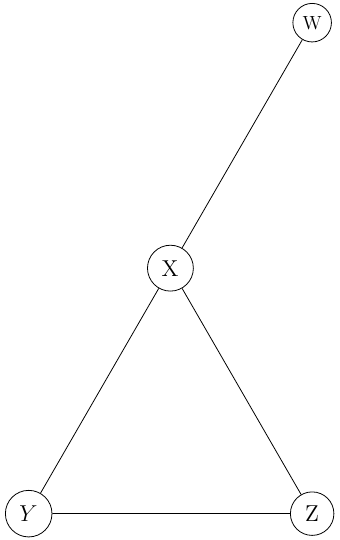}
    \end{subfigure}%
    ~ 
    \begin{subfigure}[b]{0.3\textwidth}
        \centering
        \includegraphics[scale = .4]{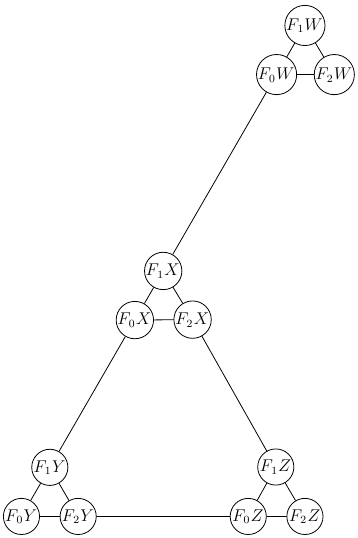}
    \end{subfigure}
    
    \caption{On the left, a general subgraph of $\Gamma_{m}$ centered for an interior vertex $X$. On the right, that subgraph extended down to $\Gamma_{m+1}$.}
\end{figure}

\noindent and similarly for all of the other vertices of degree 3. In other words, the value of the extended eigenfunction on a vertex is a function of the vertex's parent, the (different) parent of the nearest neighboring vertex, and the eigenvalue of the extended function. The $(4-\lambda_{m+1})$ term acts as a weighting factor, which makes sense as the value of the function at the parent cell should affect the value more.

To continue eigenfunctions to vertices of degree 2, we see in the above figure that if $Y$ is a vertex of degree 2 on $\Gamma_{m}$, then $F_{1}Y$ is a vertex of degree 2 on $\Gamma_{m+1}$. To extend $u$ to an eigenfunction on $\Gamma_{m+1}$, we first observe that such an eigenfunction would satisfy

$$(2-\lambda_{m+1})u'(F_{1}Y) = u'(F_{0}Y)+u'(F_{2}Y).\eqno(3.9)$$

\noindent We then add $u'(F_{1}Y)$ to both sides and apply (3.9) and solve for $u'(F_{1}Y)$ to get

$$u'(F_{1}Y) = \frac{3u(Y)}{3-\lambda_{m+1}}.\eqno(3.10)$$

\noindent This is consistent with (3.8) as every vertex adjacent to $v_{5}$ shares its parent cell, so the only factors that should matter are the current eigenvalue and the value of the function on the parent cell.

All that remains is to figure out what the $\lambda^{(k)}_{m+1}$ are. Assume $m \geq 2.$ Adding together the Laplacian relations for $v_{1}, v_{2},$ and $v_{3}$ and applying (2.2) gives

$$(1-\lambda_{m+1})3u(X) = u(v_{4}) + u(v_{7}) + u(v_{b}). \eqno(3.11)$$

\noindent Doing the same for $v_{4}, v_{7},$ and $v_{b}$, and then applying (3.6) gives

$$(2-4\lambda_{m+1} + \lambda_{m+1}^{2})3u(X) = u(v_{5}) + u(v_{6}) + u(v_{8}) + u(v_{9}) + u(v_{a}) + u(v_{c}). \eqno(3.12)$$

\noindent Adding together (3.6) and (3.7) and applying (3.1) gives

$$(3-\lambda_{m})3u(X) = (3-5\lambda_{m} + \lambda_{m}^2)3u(X) \eqno(3.13)$$

\noindent which simplifies to

$$\lambda_{m} = \lambda_{m+1}(5-\lambda_{m+1}), \eqno(3.14)$$

\noindent which has solutions

$$\lambda^{(k)}_{m+1} = \frac{5 + \varepsilon_{m+1} \sqrt{25-4\lambda_{m}}}{2} \eqno(3.15)$$

\noindent for $k = 1, 2$ where $\varepsilon = \pm 1.$ We still have the case where $m=1$ remaining. It is easy to see that the Neumann eigenbasis of $\beta_{0}$ as discussed in section 2.6 is the same as the eigenbasis of $\Gamma_{1}$ (look at Figure 2 to see the functions). One can easily verify that (3.15), (3.8), (3.10) and the analogs of the latter two formulas define continuations of eigenfunctions on $\Gamma_{1}$ to eigenfunctions on $\Gamma_{2}$ for the three basis functions, hence the formulas work for every eigenfunction on $\Gamma_{1}.$

\end{proof}
We see that (3.8), (3.10), and their analogs imply that any given continuation is unique for a given eigenvalue, and (3.15) says that there exist at most two eigenvalues. Similar to the case of the forbidden eigenvalues on $\beta_{m},$ if one of the eigenvalues given by (3.15) happens to be 3 or 5, (3.8), (3.10), and their analogs are no longer valid, so we cannot extend to eigenfunctions in these cases. Note also that (2.9) and (3.15) are the same. This will prove to be vital in the proof of the sampling theorem

\subsection{Eigenbasis of $\Gamma_{m}$}
With spectral decimation in hand, we now have all of the tools that we need in order to produce the basis. Our construction is recursive.

The eigenbasis of $\Gamma_{0}$ is obvious. As mentioned in the proof of Theorem 3.2, the eigenbasis of $\Gamma_{1}$ turns out to be the same as the Neumann eigenbasis for $\beta_{0}$ listed in Section 2. We reprint this eigenbasis in Figure 10 for convenience.

\begin{figure}
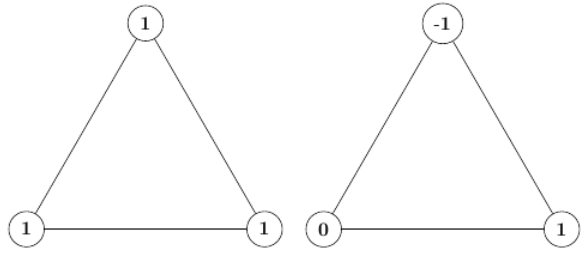

		\centering
    \begin{subfigure}[b]{0.3\textwidth}
        \centering
        \includegraphics[scale = .6]{Gamma1BasConstant.png}
    \end{subfigure}%
    ~ 
    \begin{subfigure}[b]{0.3\textwidth}
        \centering
        \includegraphics[scale = .6]{Gamma1BasNonconstant.png}
    \end{subfigure}
    
    \caption{The basis of $\Gamma_{1}$ consists of the constant eigenfunction on the left and two rotations of the eigenfunction on the right. The left eigenfunction has eigenvalue 0, whereas the one on the right has eigenvalue 3.}
\end{figure}

Now, consider $\Gamma_{2}$. We can create 5 basis elements by using the spectral decimation equations derived above to extend the three elements of the basis of $\Gamma_{1}$ down to $\Gamma_{2}$. Note that (3.15) implies that the constant eigenfunction can only bifurcate into eigenfunctions with eigenvalues 0 and 5. However, 5 is a  forbidden eigenvalue, so the constant eigenfunction can only extend to the constant eigenfunction. The remaining elements of the basis are computed by inspection, and are listed below. The first type is a "battery chain" construction around the hexagon in the graph. Start at a point that lies on the hexagon and place the value $-1$. Then go around the hexagon clockwise, alternating between placing the values $1$ and $-1$ on vertices until every point on the hexagon is nonzero. The second type is constructed by placing a $2$ at one boundary vertex, $-1$ at its adjacent vertices, $-1$ at the adjacent vertices of the boundary point's adjacent vertices, and $1$ at the two remaining non-boundary vertices. The remaining elements of the basis consist of one function of the first type, and the three rotations of the second type. Refer to Figure 11 for details.

\begin{figure}
		\centering
    \begin{subfigure}[b]{0.3\textwidth}
        \centering
        \includegraphics[scale = .6]{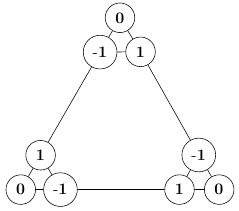}
    \end{subfigure}%
    ~ 
    \begin{subfigure}[b]{0.3\textwidth}
        \centering
        \includegraphics[scale = .6]{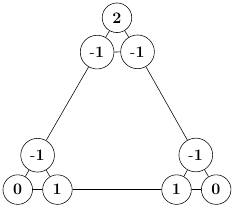}
    \end{subfigure}
    
    \caption{The two types of basis elements for $\Gamma_{2}$ that are not continued from $\Gamma_{1}$. The left has eigenvalue 5 whereas the right has eigenvalue 3.}
\end{figure}

In general, consider going from $\Gamma_{m-1}$ to $\Gamma_{m}$. Using spectral decimation, we take the eigenbasis on $\Gamma_{m-1}$ and extend it to a linear independent set in the space of functions on $\Gamma_{m}$ with cardinality $2\cdot3^{m-1}-1$ on $\Gamma_{m}$, where every eigenfunction in the basis on $\Gamma_{m-1}$ extends to two eigenfunctions on $\Gamma_{m}$ except for the constant function, which only extends to the constant function.

The case of eigenfunctions with eigenvalue 5 on $\Gamma_{m}$ is similar to the case of eigenfunctions with eigenvalue 5 on $\beta_{m}.$ For $\beta_{m},$ we placed consecutive pairs of 1 and -1 around each downward pointing triangle 1-cycle to get an eigenfunction with eigenvalue 5. It is not hard to see by the definition of $\Gamma_{m}$ that for every downward pointing triangle 1-cycle in $\beta_{m},$ there is a hexagon 1-cycle on $\Gamma_{m}$ (look at $\beta_{2}$ and $\Gamma_{2}$ to make sense of this. The converse also holds. Given the one to one correspondence of downward pointing triangle 1-cycles on $\beta_{m}$ with hexagon 1-cycles on $\Gamma_{m},$ we see that there are $1 + 3 + 3^{2} +...+3^{m-2} = \frac{3^{m-1}-1}{2}$ hexagon 1-cycles on $\Gamma_{m}$. We proceed analogously to our construction of the eigenfunction with eigenvalue 5 of $\Gamma_{2}$: pick a hexagon and a vertex on it and assign it a value of $-1$, then continue clockwise around the hexagon, alternating between $1$ and $-1$ until every vertex on the hexagon has a nonzero value, and assign 0 to the rest of the vertices. Do this for every hexagon on $\Gamma_{m}$ to get $\frac{3^{m-1}-1}{2}$ eigenfunctions with eigenvalue 5.

We now consider the eigenfunctions with eigenvalue 3. This case is similar to that of the eigenfunctions with eigenvalue 6 on $\beta_{m}.$ First, we take each vertex on the boundary of $\Gamma_{m}$ and copy the appropriate rotation of the second type of non-extended eigenfunction on $\Gamma_{2}$ by assigning 2 to a boundary point, then copying the remaining portion of the eigenfunctions with eigenvalue 3 on $\Gamma_{2}.$. This yields three eigenfunctions. To get the remaining eigenfunctions, we first note that a routine induction argument shows that, for $m \geq 2,$ there are $3^{m-2}$ copies of $\Gamma_{2}$ in $\Gamma_{m}$. Take two adjacent copies of $\Gamma_{2}$ in $\Gamma_{m},$ and consider the edge connecting them. Assign a value of $2$ to each of the vertices on the edge, and then repeat the construction of the second type of nonextended eigenfunction on $\Gamma_{2}$ twice. Refer to Figure 12 below for the specific construction. An inductive argument shows that there are $\frac{3^{m-1}-3}{2}$ eigenfunctions of this type, so we have a total of $\frac{3^{m-1}+3}{2}$ eigenfunctions at level $m$ with eigenvalue 3.

\begin{figure}
\begin{center}
\includegraphics[scale=.7]{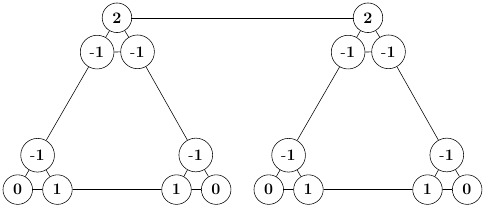}
\end{center}
\caption{The type of eigenfunctions with eigenvalue 3 on level $m$ that have no support on the boundary}
\end{figure}

Counting the eigenfunctions that we have thus far, we see we have

$$2\cdot3^{m-1}-1 + \frac{3^{m-1}-1}{2} + \frac{3^{m-1}+3}{2} = 3^{m}.$$

\noindent However, we can have at most $3^{m}$ linearly independent basis elements, so we have necessarily constructed the eigenbasis.

\section{The Sampling Theorem}

Recall the construction of the first $3^{m}$ elements of the Neumann eigenbasis of $SG$ as mentioned in Section 2.7. We can now precisely define what we mean by a bandlimited function on $SG.$

\begin{defn}
A \emph{level $m$ bandlimited function} on $SG$ is a function that is a linear combination of the first $3^{m}$ elements of the Neumann eigenbasis of $SG$ (see Section 2.7).
\end{defn}

\noindent By the discussion in Section 2.7, we know that every level $m$ bandlimited function is the linear combinations of eigenfunctions with eigenvalue at most $B 5^m.$ Corresponding to the discrete and continuous averages $A_{C}$ and $B_{C}$ defined in Section 2.2, we define maps $A_{m}$ and $B_{m}$ from the space of real valued functions $u$ on $SG$ to functions on $\Gamma_{m}$ via the formulas

$$A_{m}(u)(x) = A_{C}(u) \eqno(4.1)$$

$$B_{m}(u)(x) = B_{C}(u) \eqno(4.2)$$

\noindent where $x$ is the vertex in $\Gamma_{m}$ corresponding to the $m$-cell $C.$ Note that $A_{m}$ is also valid as a linear map from functions on $\beta_{m}$ to functions on $\Gamma_{m}.$

We are now ready to state the sampling theorem mentioned in Section 1.

\begin{thm}
On $SG$, a level $m$ bandlimited function $u$ is uniquely determined by $A_{m}(u)$ and $B_{m}(u).$
\end{thm}

Right away, we see the need for restricting to the first $3^{m}$ Neumann eigenfunctions on $SG$ in our definition of level $m$ bandlimited functions. If $u$ is an eigenfunction on $SG$ with $l(u) = m$ and $E(u) = 6,$ then $A_{m}(u)$ is the zero function (see the discussion in Section 2 for the construction of the eigenfunctions with eigenvalue 6). Thus any information from eigenfunctions with eigenvalue 6 is lost in the averaging.

In order to prove the theorem, we need a technical lemma.

\begin{lem}
Let $m \geq 1,$ let $u$ be an eigenfunction on $\beta_{m-1}$ with eigenvalue $\lambda_{m-1}$ , and let $u'$ be its continuation by spectral decimation to an eigenfunction on $\beta_{m}$ with eigenvalue $\lambda_{m}$. Then $A_{m}(u')$ is an eigenfunction on $\Gamma_{m}$ with eigenvalue $\lambda_{m}$.
\end{lem}

\begin{proof}
The idea of the proof is straightforward: take $u$ and continue it to $u'$ via spectral decimation. Every point in $V_{m} \backslash V_{m-1}$ must satisfy the formulas outlined in (2.10)-(2.12). We define a map $A'_{m}$ analogous to that defined in (4.1) that takes functions $f: \beta_{m} \to \mathbb{R}$ to functions $A'_{m}(u) : \Gamma_{m} \to \mathbb{R}$ such that if $x$ is the vertex on $\Gamma_{m}$ corresponding to the $m$-cell $C,$ then

$$A'_{m}(f)(x) = A_{C}(u). \eqno(4.3)$$

\noindent We then compute $A'_{m}(u')$ and show that it is an eigenfunction by considering separately the cases where a given vertex has degree 3 and degree 2, making use of the eigenrelation on $\beta_{m-1}$ and the fact that $\lambda_{m-1} = \lambda_{m}(5-\lambda_{m}).$ 

There are two cases to consider: vertices on $\Gamma_{m}$ with degree 3 (which occur for $m \geq 2$) and vertices on $\Gamma_{m}$ with degree 2 (which occur for $m \geq 1$). We first consider the degree 3 case. Refer to Figure 13.

\begin{figure}
\begin{center}
\includegraphics[scale=.7]{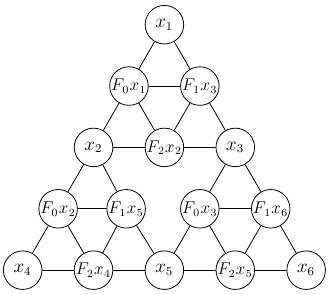}
\end{center}
\caption{A general $(m-2)$ cell of $\beta_{m}$ for $m \geq 2.$}
\end{figure}

Let $C_{1}$ be the $m$-cell with vertices $x_{1}, F_{0}x_{1}, F_{1}x_{3},$ let $C_{2}$ be the $m$-cell with vertices $F_{0}x_{1}, x_{2}, F_{2}x_{2},$ let $C_{3}$ be the $m$-cell with vertices $F_{1}x_{3}, F_{2}x_{2}, x_{3},$ and let $Z$ be the $m$-cell with vertices $x_{2}, F_{0}x_{2}, F_{1}x_{5}.$ Let $f = A_{m}(u').$ We need to show

$$(3-\lambda_{m})f(C_{2}) = f(C_{1})+f(C_{3})+f(Z).$$

\noindent By the definition of $f,$ using (2.10)-(2.12), we have

\begin{align*}
f(C_{i})&=\frac{(2-\lambda_{m})(5-\lambda_{m})u(x_{i}) + (4-\lambda_{m})(2u(x_{i})+u(x_{i+1})+u(x_{i+2}))}{3(2-\lambda_{m})(5-\lambda_{m})}\\
&+ \frac{2(u(x_{i+1})+u(x_{i+2}))}{3(2-\lambda_{m})(5-\lambda_{m})},
\end{align*}

\noindent where addition of indices is cyclic, and
\begin{align*}
f(Z)&=\frac{(2-\lambda_{m})(5-\lambda_{m})u(x_{2}) + (4-\lambda_{m})(2u(x_{2})+u(x_{4})+u(x_{5}))}{3(2-\lambda_{m})(5-\lambda_{m})}\\
&+ \frac{2(u(x_{4})+u(x_{5}))}{3(2-\lambda_{m})(5-\lambda_{m})}
\end{align*}

\noindent Thus, we have

\begin{align*}
f(C_{1})+f(C_{3})+f(Z)&=\frac{((2-\lambda_{m})(5-\lambda_{m})+(8-2\lambda_{m}))(\sum_{i=1}^{3} u(x_{i}))}{3(2-\lambda_{m})(5-\lambda_{m})}\\
&+\frac{(4-\lambda_{m})(u(x_{2})+ \sum_{j=1}^{5}u(x_{j}))}{3(2-\lambda_{m})(5-\lambda_{m})}\\
&+ \frac{2(u(x_{2})+ \sum_{j=1}^{5}u(x_{j}))}{3(2-\lambda_{m})(5-\lambda_{m})}
\end{align*}
\noindent Since $u$ is an eigenfunction on $\beta_{m-1},$ we have that

$$(4-\lambda_{m-1})u(x_{2}) = u(x_{1})+u(x_{3})+u(x_{4})+u(x_{5}) \eqno(4.4)$$

\noindent Using (4.4) and the fact that 

$$(2-\lambda_{m})(5-\lambda_{m})+(8-2 \lambda_{m}) = (3-\lambda_{m})(6-\lambda_{m}),$$ 

\noindent we see that

\begin{align*}
f(C_{1})+f(C_{3})+f(Z)&=\frac{(3-\lambda_{m})(6-\lambda_{m})(\sum_{i=1}^{3}u(x_{i}))}{3(2-\lambda_{m})(5-\lambda_{m})}\\
&+\frac{(6-\lambda_{m})(6-\lambda_{m-1})u(x_{2})}{3(2-\lambda_{m})(5-\lambda_{m})}\\
\end{align*}

\noindent Now, recalling that $\lambda_{m-1} = \lambda_{m}(5-\lambda_{m})$ by (2.9), we see that

$$(6-\lambda_{m})(6-\lambda_{m-1}) = (6-\lambda_{m})(3-\lambda_{m})(2-\lambda_{m}), \eqno(4.5)$$

\noindent thus
\begin{align*}
f(C_{1})+f(C_{3})+f(Z)&=\frac{(3-\lambda_{m})(6-\lambda_{m})(u(x_{1})+u(x_{3}))}{3(2-\lambda_{m})(5-\lambda_{m})}\\
&+\frac{(6-\lambda_{m})(3-\lambda_{m})^{2}u(x_{2})}{3(2-\lambda_{m})(5-\lambda_{m})}\\
\end{align*}

\noindent Observing the algebraic identity

$$(6-\lambda_{m})(3-\lambda_{m})^{2} = (2-\lambda_{m}(3-\lambda_{m})(5-\lambda_{m}) + (3-\lambda_{m})(8-2 \lambda_{m}),$$

\noindent we factor out a $(3-\lambda_{m})$ to obtain

\begin{align*}
f(C_{1})+f(C_{3})+f(Z)&= (3-\lambda_{m}) \left(\frac{(6-\lambda_{m})(u(x_{1})+u(x_{3}))}{3(2-\lambda_{m})(5-\lambda_{m})}\right)\\
&+ (3-\lambda_{m}) \left(\frac{(2-\lambda_{m})(5-\lambda_{m})u(x_{2})}{3(2-\lambda_{m})(5-\lambda_{m})} \right)\\
&+ (3-\lambda_{m}) \left(\frac{(4-\lambda_{m})(2u(x_{2}))}{3(2-\lambda_{m})(5-\lambda_{m})} \right) \\
&=(3-\lambda_{m})f(C_{2})
\end{align*}

For the degree 2 case, consult Figure 14.

\begin{figure}
    \centering
\includegraphics[width =.3\textwidth]{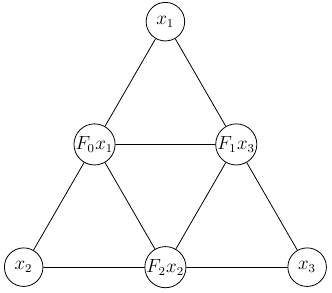}
\caption{A general $(m-1)$-cell of $\beta_{m}.$}
\end{figure}

We let $C_{i}$ be the $m$-cell containing the vertex $x_{i}.$ We need to show that

$$(2-\lambda_{m})f(C_{1}) = f(C_{2})+f(C_{3}).$$

Again, using (2.10)-(2.12), we have

\begin{align*}
f(C_{i})&=\frac{(2-\lambda_{m})(5-\lambda_{m})u(x_{i}) + (4-\lambda_{m})(2u(x_{i})+u(x_{i+1})+u(x_{i+2}))}{3(2-\lambda_{m})(5-\lambda_{m})}\\
&+ \frac{2(u(x_{i+1})+u(x_{i+2}))}{3(2-\lambda_{m})(5-\lambda_{m})}
\end{align*}

\noindent where addition of indices is cyclic. From this, we have

\begin{align*}
f(C_{2})+f(C_{3})&=\frac{((2-\lambda_{m})(5-\lambda_{m})+(8-2 \lambda_{m})+(6-\lambda_{m}))(u(x_{2})+u(x_{3}))}{3(2-\lambda_{m})(5-\lambda_{m})}\\
&+\frac{2(6-\lambda_{m})u(x_{1})}{3(2-\lambda_{m})(5-\lambda_{m})}\\
&=\frac{(6-\lambda_{m})(4-\lambda_{m})(u(x_{2})+u(x_{3}))+2(6-\lambda_{m})u(x_{1})}{3(2-\lambda_{m})(5-\lambda_{m})}\\
&=\frac{(6-\lambda_{m})(2u(x_{2})+2(u(x_{3})) + (2-\lambda_{m})(6-\lambda_{m})(u(x_{2})+u(x_{3}))}{3(2-\lambda_{m})(5-\lambda_{m})}\\
&+\frac{2(6-\lambda_{m})u(x_{1})}{3(2-\lambda_{m})(5-\lambda_{m})}
\end{align*}

\noindent From the Neumann condition (2.13), we have

$$(4-\lambda_{m-1})u(x_{1}) = 2u(x_{2})+2u(x_{3})$$

\noindent and since 

$$(6-\lambda_{m})(4-\lambda_{m-1})+2(6-\lambda_{m}) = (6-\lambda_{m})(4-\lambda_{m}(5-\lambda_{m}))+2(6-\lambda_{m})$$

\begin{align*}
(6-\lambda_{m})(2-\lambda_{m})(3-\lambda_{m})&=(2-\lambda_{m})(18-9 \lambda_{m} + \lambda_{m}^{2}) \\
&= (2-\lambda_{m})((2-\lambda_{m})(5-\lambda_{m})+2(4-\lambda_{m}))
\end{align*}

\noindent we see that

$$f(Y)+f(Z) = (2-\lambda_{m})f(X)$$

\noindent as desired.
\end{proof}
We have the following important corollary

\begin{cor}
The space $P_{m}$ of real-valued functions on $\beta_{m}$ spanned by the Neumann eigenfunctions on $\beta_{m}$ that do not have eigenvalue 6 is isomorphic as a vector space over $\mathbb{R}$ to the space of real-valued functions on $\Gamma_{m},$ with $A_{m}$ being the isomorphism. Furthermore, $A_{m}$ preserves eigenvalue in the sense that eigenfunctions on $\beta_{m}$ with eigenvalue $\lambda_{m}$ are mapped to eigenfunctions on $\Gamma_{m}$ with eigenvalue $\lambda_{m}.$
\end{cor}
\begin{proof}
We proceed by induction on $m.$ The case $m = 0$ is trivial. For the case $m=1,$ $P_{1}$ has a basis consisting of the constant eigenfunction, as well as the two eigenfunctions with eigenvalue 3 obtained by continuing the eigenbasis elements with eigenvalue 6 on $\beta_{0}$ mentioned in Section 2.6 via the spectral decimation formulas (2.9)-(2.12). Clearly $A_{1}$ takes the constant eigenfunction to the constant eigenfunction, and one can easily compute that the other two Neumann eigenbasis elements with eigenvalue 3 on $\beta_{1}$ map under $A_{1}$ to constant multiples of the eigenbasis elements with eigenvalue 3 on $\Gamma_{1}$ given in Section 3.3. By linearity, it follows that $A_{1}$ is an isomorphism.

Now, let $m \geq 2,$ and assume that $A_{m-1}$ is an isomorphism. It is easy to see that $A_{m}$ maps the Neumann eigenbasis elements with eigenvalue 5 on $\beta_{m}$ in a 1-1 correspondence with scalar multiples of the eigenbasis elements with eigenvalue 5 on $\Gamma_{m}.$ Similarly, $A_{m}$ maps the Neumann eigenbasis elements with eigenvalue 3 on $\beta_{m}$, obtained by applying the spectral decimation formulas (2.9)-(2.12) to continue the eigenbasis elements with eigenvalue 6 on $\beta_{m-1}$ down to eigenfunctions with eigenvalue 3 on $\beta_{m},$ in a 1-1 correspondence with scalar multiples of the eigenbasis elements with eigenvalue 3 on $\Gamma_{m}.$

Now, consider an eigenbasis element $g$ on $\Gamma_{m}$ with eigenvalue $\lambda_{m} \neq 3, 5.$ By the discussion in Section 3.3, we can find some eigenfunction $f$ on $\Gamma_{m-1}$ with eigenvalue $\lambda_{m-1}$ such that $g$ is the continuation of $f$ via spectral decimation given by Theorem 3.2. By the induction hypothesis, there is some Neumann eigenfunction $u$ on $\beta_{m-1}$ with eigenvalue $\lambda_{m-1}$ such that $A_{m-1}(u) = f;$ note that since $A_{m-1}$ is an isomorphism, the values of $u$ are determined by the values of $f.$ Let $u'$ be the continuation of $u'$ to $u$ via spectral decimation to an eigenfunction with eigenvalue $\lambda_{m}$ (such a continuation exists since the relations (2.9) and (3.15) are the same). Since the values of $u$ on a vertex in $\beta_{m-1}$ are determined by the values of $f,$ so are the values of $u'.$ By Lemma 4.3, we know that $A_{m}(u')$ is an eigenfunction on $\Gamma_{m}$ with eigenvalue $\lambda_{m}.$ Since the values of $u'$ depend on $f,$ so do the values of $A_{m}(u').$ Thus $A_{m}(u')$ is a continuation of the eigenfunction $f$ on $\Gamma_{m-1}$ to an eigenfunction on $\Gamma_{m}.$ Furthermore, by Theorem 3.2, such a continuation is unique. Since $g$ is such a continuation, we have that $A_{m}(u') = g.$

Thus for every element in the eigenbasis of $\Gamma_{m}$ discussed in Section 3.3, we can find an element $u \in P_{m}$ such that $A_{m}(u) = g.$ By the linearity of $A_{m},$ we can conclude that $A_{m}$ is a surjection. Since $P_{m}$ and the space of functions on $\Gamma_{m}$ both have dimension $3^{m},$ we can conclude that $A_{m}$ is an isomorphism between these two spaces. Furthermore, by the above discussion, we know that $A_{m}$ preserves eigenvalues.
\end{proof}

We now proceed with the proof of Theorem 4.2. Let $u$ be a level $m$ bandlimited function on $SG,$ so

$$u = \sum_{i=1}^{3^{m}} c_{i} u_{i},$$

\noindent where the $u_{i}$ are the first $3^{m}$ elements of the Neumann eigenbasis of $SG.$ The restriction of each $u_{i}$ to $\beta_{m},$ denoted $\left. u_{i} \right|_{\beta_{m}}$ is an element of the eigenbasis of the space $P_{m}$ mentioned in the statement of Corollary 4.4. By Corollary 4.4, the values of these restrictions are uniquely determined by their averages $A_{m}(\left. u_{i} \right|_{\beta_{m}}).$ By definition, we have

$$A_{m}(\left. u_{i} \right|_{\beta_{m}}) = A_{m}(u_{i}),$$

so the values of $\left. u_{i} \right|_{\beta_{m}}$ are determined uniquely by the discrete average $A_{m}(u_{i}).$ However, by the definition of the first $3^{m}$ elements of the Neumann eigenbasis of $SG,$ each $u_{i}$ is obtained by continuing down $\left. u_{i} \right|_{\beta_{m}}$ via spectral decimation in a unique way. Hence each $u_{i}$ is determined uniquely by $A_{m}(u_{i}).$ By the definition of $u$ and the linearity of $A_{m}$ as a map of real valued functions on $SG$ to real valued functions on $\Gamma_{m},$ we can conclude that $u$ is determined uniquely by $A_{m}(u).$

\begin{figure}
		\centering
    \begin{subfigure}[b]{0.3\textwidth}
        \centering
        \includegraphics[scale = .4]{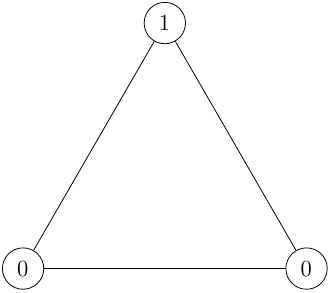}
    \end{subfigure}%
    ~ 
    \begin{subfigure}[b]{0.3\textwidth}
        \centering
        \includegraphics[scale = .4]{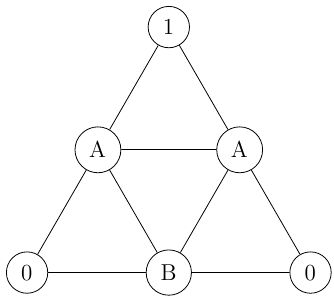}
    \end{subfigure}
    
    \caption{On the left, the function $u$ on the $k$-cell $C$ mentioned in the proof of Equation (4.6). On the right, the extension of $u$ on $C$ via spectral decimation. Here, we have $A = \frac{4-\lambda_{i}^{(k+1)}}{(2- \lambda_{i}^{(k+1)})(5 - \lambda_{i}^{(k+1)})}$ and $B = \frac{2}{(2- \lambda_{i}^{(k+1)})(5 - \lambda_{i}^{(k+1)})}.$
}
\end{figure}

To prove the theorem for the continuous averages $B_{m},$ we note that there is a simple linear relationship between the discrete and continuous averages. Suppose $u_{i}$ is one of the first $3^{m}$ Neumann eigenfunctions on $SG$ with eigenvalue $\lambda_{i}$ born on level $m_{0} \le m.$ Then there exists a sequence $\lambda_{i}^{(k)}$ for $k \geq m_{0}$ with $\lambda_{i} = \lim_{k \to \infty} \frac{3}{2}5^{k} \lambda_{i}^{(k)}$ and the restriction of $u_{i}$ to $\beta_{k+1}$ is obtained from the restriction of $u_{i}$ to $\beta_{k}$ by spectral decimation. If $C$ denotes a cell of level $k$ with $C = C_{1} \cup C_{2} \cup C_{3}$ its decomposition into three cells of level $k+1,$ we claim

$$A_{C}(u_{i}) = \frac{\left( 1 - \frac{\lambda_{i}^{(k+1)}}{6} \right)}{3 \left(1 - \frac{\lambda_{i}^{(k+1)}}{2} \right)} \left(A_{C_{1}}(u_{i}) + A_{C_{2}}(u_{i})+ A_{C_{3}}(u_{i}) \right) \eqno(4.6)$$

\noindent It suffices to verify this for the case when $u$ on $C$ has values boundary values shown in the left portion of Figure 15, whose boundary values on $C_{1}, C_{2},$ and $C_{3}$ are given by spectral decimation and shown in right portion of Figure 15. Then (4.6) follows from a direct computation. By linearity and symmetry, (4.6) holds for all boundary values on $C.$

By iteration and passing to the limit we obtain

$$B_{C}(u_{i}) = \frac{\prod_{k > m} \left(1 - \frac{ \lambda_{i}^{(k)}}{6} \right)}{\prod_{k > m} \left(1 - \frac{ \lambda_{i}^{(k)}}{2} \right)} A_{C}(u_{i}) \eqno(4.7)$$

\noindent for any $m$-cell $C,$ since $B_{C}$ is the limit of discrete averages over boundaries of $k$-cells in $C.$ We denote the coefficients in (4.7) by $b_{i},$ and note that they are nonzero. If we enumerate the $3^{m}$ $m$-cells by $C_{n},$ then the theorem for $A_{m}$ is equivalent to the invertibility of the $3^{m} \times 3^{m}$ matrix $(a_{in}) = (A_{C_{n}}(u_{i})),$ and the theorem for $B_{m}$ is equivalent to the invertibility of the matrix $(b_{im}) = (B_{C_{m}}(u_{i})).$ Since one matrix is obtained from the other by multiplication of each row by a nonzero constant, the span of the rows is the same, so the result for $A_{m}$ implies the result for $B_{m}.$
\section{Cardinal Interpolant Functions}
\subsection{Numerical Data on $SG$}
We recall that, for a set $A$, $\{\psi_{y} \vert y \in A\}$ is a set of cardinal interpolants for $A$ if

$$\psi_{y}(x) = \delta(y,x)$$

\noindent for all $x \in A$. In this paper, we are interested in cardinal interpolants for $m$-cells. Specifically, we say that $\psi$ is a cardinal interpolant function for the $m$-cell $F_{w}SG$ if the $B_{m}(\psi)$ equals 1 on the vertex of $\Gamma_{m}$ corresponding to the word $w$ and equals 0 for all other vertices. By the theorem in the previous section, we can take such functions to be bandlimited, so we can restrict ourselves to looking at functions on the $\Gamma_{m}.$ We call these bandlimited cardinal interpolant functions on $m$-cells {\bf{sampling functions}}. Figures 16-24 show all of the sampling functions for $\Gamma_{1}, \Gamma_{2}, \Gamma_{3},$ and $\Gamma_{4}.$

Having performed some numerical calculations, we can make two conjectures.

\begin{con}
The sampling functions are uniformly bounded by some constant (around 1.3) for all $m$.
\end{con}

\noindent This conjecture matches the case of cardinal interpolants on the real line, which are translates of the $\sinc$ function. The following conjecture, however, does not.

\begin{con}
Let $\psi_{m}$ be a sampling function on level $m$. Then

$$\int_{SG} \left| \psi_{m} \right|^{2} \le c3^{-m}$$

\noindent where $c$ is a constant around 1.2.
\end{con}

In Table 1, we present numerical evidence for these conjectures.

We can compare these sampling functions on cells with the sampling functions on vertices discussed in \cite{oberlin2003sampling}. First, we remark that \cite{oberlin2003sampling} used a different notion of bandlimited, essentially using Definition 4.1 but instead focusing on the first $\frac{3^{m+1}-3}{2}$ Dirichlet eigenfunctions instead of the first $3^{m}$ Neumann eigenfunctions. In the following we refer to the notion of bandlimited in \cite{oberlin2003sampling} as vertex-bandlimited while we continue to refer to the case of Definition 4.1 as bandlimited.

With that in mind, we say that a cardinal interpolant function for $\beta_{m}$ is a function on $SG$ that is equal to 1 for a non-boundary vertex on $\beta_{m}$ and zero for every other vertex on $\beta_{m}.$ We then say a vertex sampling function on $\beta_{m}$ is a cardinal interpolant function on $\beta_{m}$ that is vertex-bandlimited. For small $m$, the vertex sampling functions have large oscillations throughout $SG,$ while the cell sampling functions do not. For larger $m$, however, we observe that both the vertex sampling functions and the cell sampling functions have the same basic oscillatory behavior: there is a big sinusoidal oscillation centered at the chosen point/cell which rapidly dies off. We have the following conjecture that is analogous to Conjecture 3.1 in \cite{oberlin2003sampling}.

\begin{con}
Let $\psi_{C}$ be the sampling function for $m$-cell $C.$ there exist constants $\alpha \approx \frac{1}{3}$ and $D,$ both independent of $C$ and $m,$ such that

$$|\psi_{C}(x)| \le D \alpha^{d_{m}(C,C_{x})}$$

\noindent where $d_{m}(C,C_{x})$ is the length of the shortest path between the vertices on $\Gamma_{m}$ that correspond to $C$ and $C_{x}.$
\end{con}

Refer to Table 2 for supporting data.
\begin{center}

  \begin{tabular}{ | l | l | l | p{5cm} |}
  \hline
  Word & m & $\max(\left| \psi \right|)$ & $3^{m}\int \left|\psi\right|^{2}$ \\ \hline
  [0] & 1 & 1.25658977540316 & 1.07842219640964 \\ \hline
  [0,0] & 2 & 1.26695538950498 & 1.07842219640964 \\ \hline
  [0,1] & 2 & 1.08571404036364 & 1.1210553607424\\ \hline
  [0,0,0] & 3 & 1.26692412467287 & 1.07839115312062\\ \hline
  [0,0,1] & 3 & 1.08594427634807 & 1.12090225373618\\ \hline
  [0,1,0] & 3 & 1.07899763862335 & 1.12043223136981 \\ \hline
  [0,1,1] & 3 & 1.07930676762662 & 1.12040455179345 \\ \hline
  [0,1,2] & 3 & 1.07936645775999 & 1.12032092081025 \\ \hline
  [0,0,0,0] & 4 & 1.26655024495913 & 1.07827446772558 \\ \hline
  [0,0,0,1] & 4 & 1.08546096243966 & 1.12070196936732 \\ \hline
  [0,0,1,0] & 4 & 1.07852883641541 & 1.12023075543376 \\ \hline
  [0,0,1,1] & 4 & 1.07882624003117 & 1.12020209908214 \\ \hline
  [0,0,1,2] & 4 & 1.07889336080433 & 1.12011942546551 \\ \hline
  [0,1,0,0] & 4 & 1.07880740875524 & 1.12019431284607 \\ \hline
  [0,1,0,1] & 4 & 1.07889107714326 & 1.12011628070581 \\ \hline
  [0,1,0,2] & 4 & 1.07889107441474 & 1.12011623345686 \\ \hline
  [0,1,1,0] & 4 & 1.07889095592936 & 1.12011621816689 \\ \hline
  [0,1,1,1] & 4 & 1.07880946182471 & 1.1201933422404 \\ \hline
  [0,1,1,2] & 4 & 1.07889096832845 & 1.12011621162118 \\ \hline
  [0,1,2,0] & 4 & 1.07889095912966 & 1.1201161798206 \\ \hline
  [0,1,2,1] & 4 & 1.07889097425727 & 1.12011622052384 \\ \hline
  [0,1,2,2] & 4 & 1.07880950409159 & 1.12019329618038 \\ \hline  
  \end{tabular}
  \captionof{table}{Numerical data for cardinal interpolants on $\Gamma_{1}, \Gamma_{2}, \Gamma_{3},$ and $\Gamma_{4}.$ Note that this information supports Conjectures 5.1 and 5.2.}
\end{center}

\begin{center}
\begin{tabular}{l|r}
Distance from $[1 2 2]$ & Bound for $|\psi_{C}|$ \\
\hline
1 & $6.079 \times 10^{-1}$ \\
2 & $1.002 \times 10^{-1}$ \\
3 & $1.919 \times 10^{-2}$ \\
4 & $5.072 \times 10^{-3}$ \\
5 & $1.087 \times 10^{-3}$ \\
6 & $3.547 \times 10^{-4}$ \\
7 & $8.403 \times 10^{-5}$ \\
\end{tabular}
\captionof{table}{Numerical data for the sampling function for the 3-cell corresponding to the word $[1 2 2].$ The table gives bounds for the sampling function depending on the distance between cells as mentioned in Conjecture 5.3}
\end{center}

\begin{figure}
    \centering
    \begin{subfigure}[b]{0.33\textwidth}
        \centering
        \includegraphics[width=\textwidth]{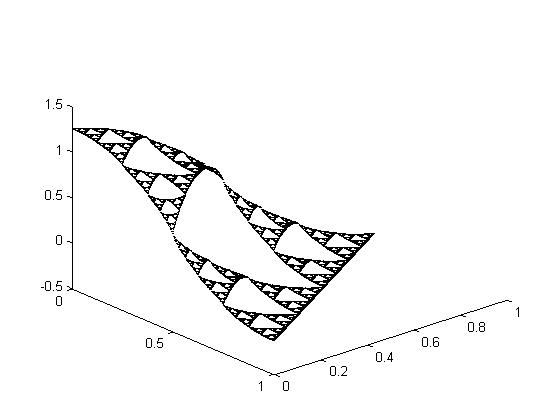}
    \end{subfigure}%
    ~ 
    \begin{subfigure}[b]{0.33\textwidth}
        \centering
        \includegraphics[width=\textwidth]{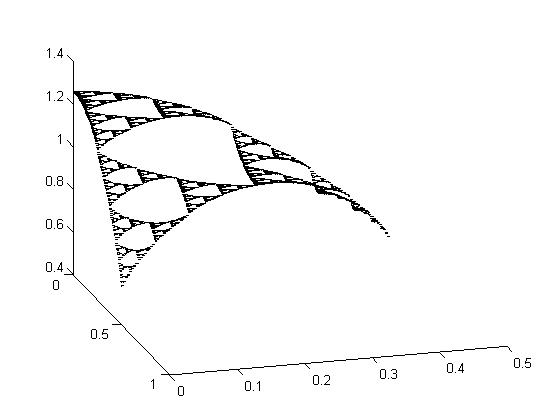}
    \end{subfigure}
    
    \caption{The unique (up to $D_{3}$ symmetries) level 1 sampling function. The bottom row features zoomed in versions of the functions on the top row.}
\end{figure}

\begin{figure}
    \centering
    \begin{subfigure}[b]{0.33\textwidth}
        \centering
        \includegraphics[width=\textwidth]{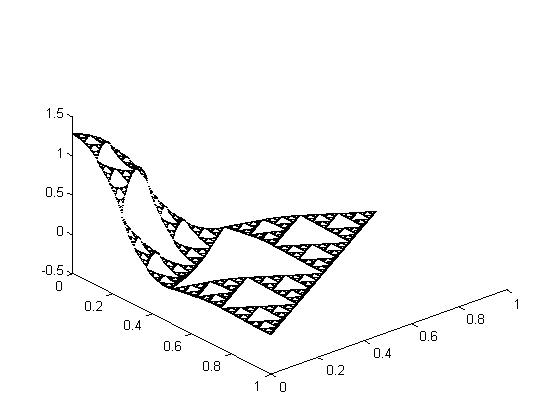}
        
    \end{subfigure}%
    ~ 
    \begin{subfigure}[b]{0.33\textwidth}
        \centering
        \includegraphics[width=\textwidth]{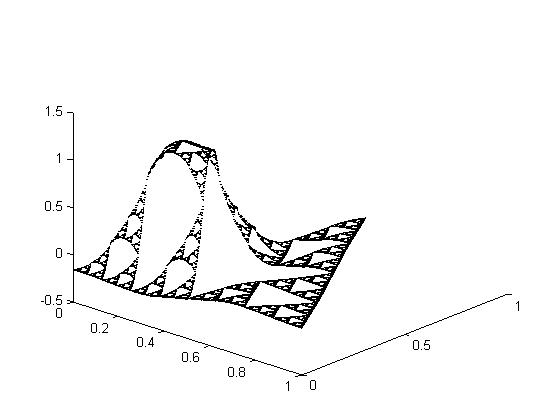}
    \end{subfigure}%

    \begin{subfigure}[b]{0.33\textwidth}
        \centering
        \includegraphics[width=\textwidth]{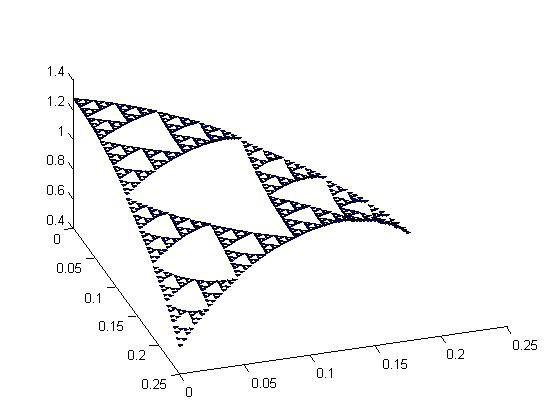}
        
    \end{subfigure}
		~
    \begin{subfigure}[b]{0.33\textwidth}
        \centering
        \includegraphics[width=\textwidth]{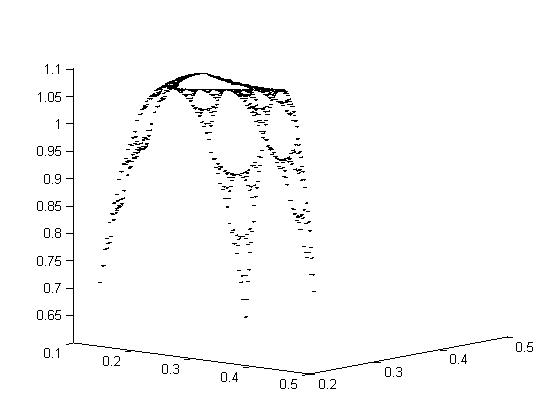}
    \end{subfigure}
    
    \caption{The unique (up to $D_{3}$ symmetries) level 2 sampling functions. The bottom row features zoomed in versions of the functions on the top row.}
\end{figure}

\begin{figure}
    \centering
    \begin{subfigure}[b]{0.33\textwidth}
        \centering
        \includegraphics[width=\textwidth]{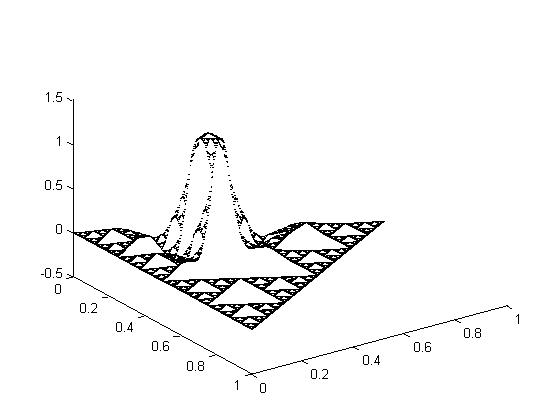}
        
    \end{subfigure}%
    ~ 
    \begin{subfigure}[b]{0.33\textwidth}
        \centering
        \includegraphics[width=\textwidth]{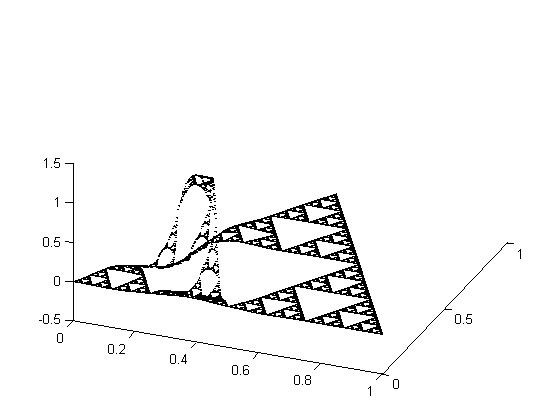}
    \end{subfigure}%

    \begin{subfigure}[b]{0.33\textwidth}
        \centering
        \includegraphics[width=\textwidth]{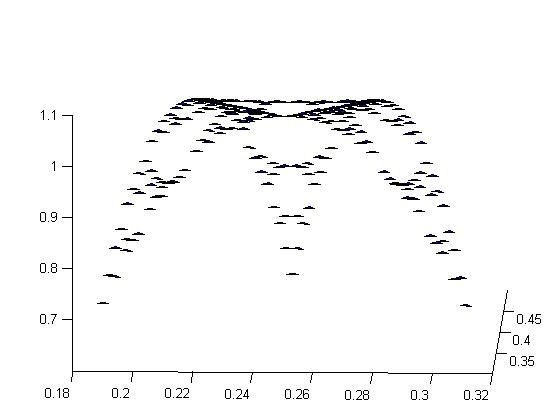}
        
    \end{subfigure}
		~
    \begin{subfigure}[b]{0.33\textwidth}
        \centering
        \includegraphics[width=\textwidth]{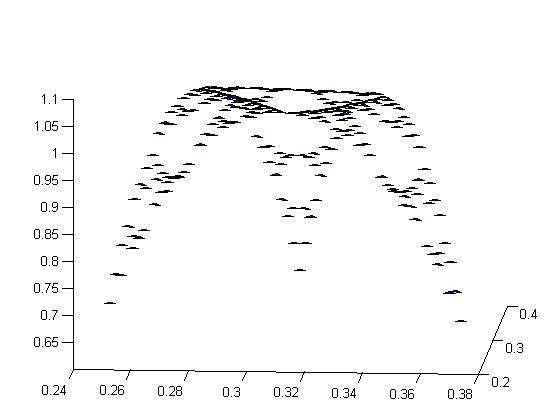}
    \end{subfigure}
    
    \caption{Some of the level 3 sampling functions.  The bottom row features zoomed in versions of the functions on the top row.}
\end{figure}

\begin{figure}
    \centering
    \begin{subfigure}[b]{0.33\textwidth}
        \centering
        \includegraphics[width=\textwidth]{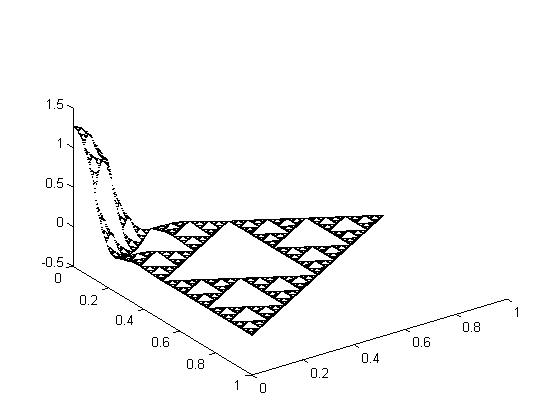}
        
    \end{subfigure}%
    ~ 
    \begin{subfigure}[b]{0.33\textwidth}
        \centering
        \includegraphics[width=\textwidth]{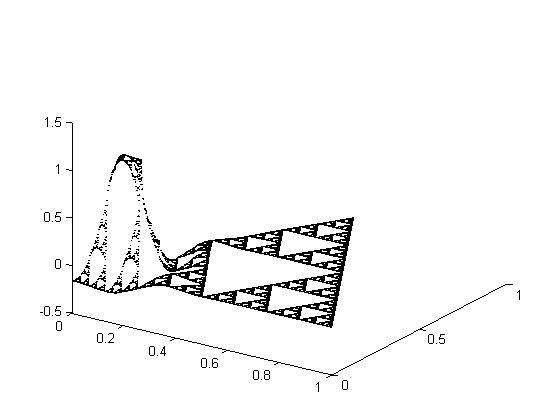}
    \end{subfigure}%
    ~
    \begin{subfigure}[b]{0.33\textwidth}
        \centering
        \includegraphics[width=\textwidth]{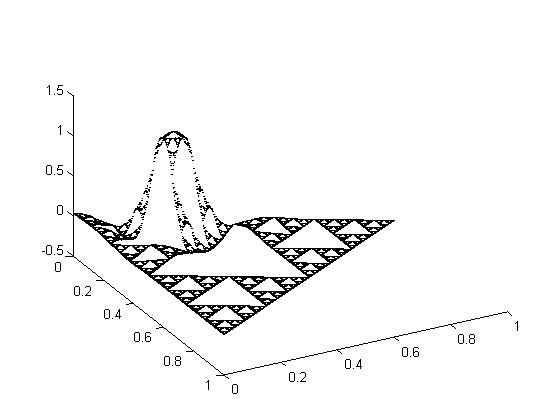}
    \end{subfigure}%
    
    \begin{subfigure}[b]{0.33\textwidth}
        \centering
        \includegraphics[width=\textwidth]{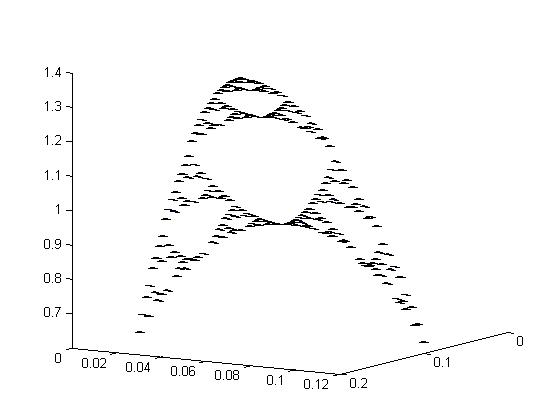}
        
    \end{subfigure}
    ~
    \begin{subfigure}[b]{0.33\textwidth}
        \centering
        \includegraphics[width=\textwidth]{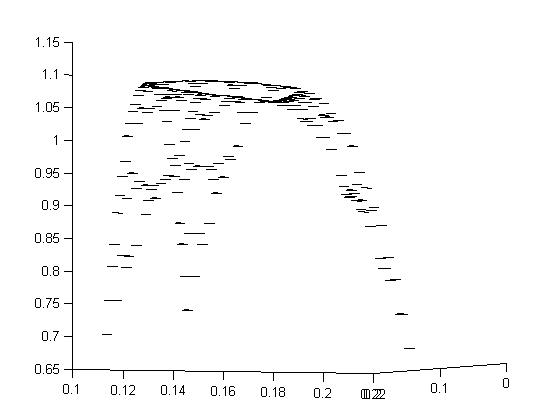}
    \end{subfigure}
    ~ 
    \begin{subfigure}[b]{0.33\textwidth}
        \centering
        \includegraphics[width=\textwidth]{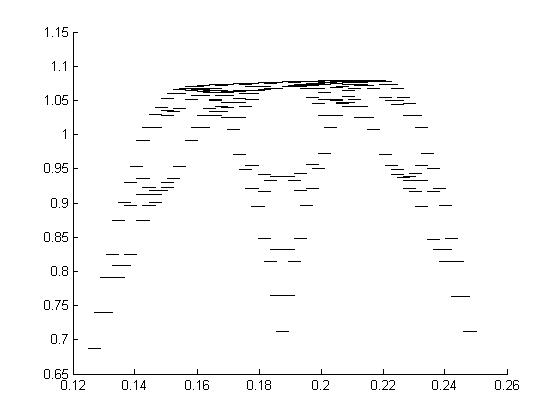}
    \end{subfigure}
    \caption{The remaining level 3 sampling functions (up to $D_{3}$ symmetries). The bottom row features zoomed in versions of the functions on the top row.}
\end{figure}

\begin{figure}
    \centering
    \begin{subfigure}[b]{0.33\textwidth}
        \centering
        \includegraphics[width=\textwidth]{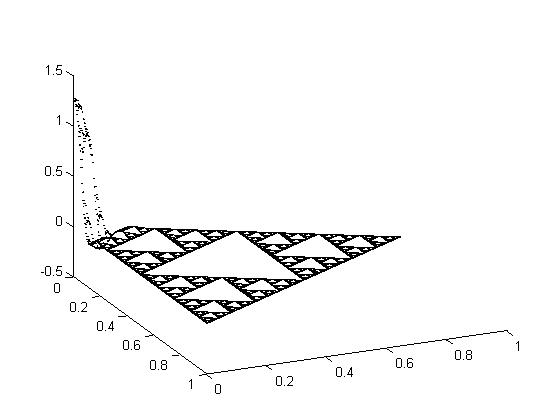}
        
    \end{subfigure}%
    ~ 
    \begin{subfigure}[b]{0.33\textwidth}
        \centering
        \includegraphics[width=\textwidth]{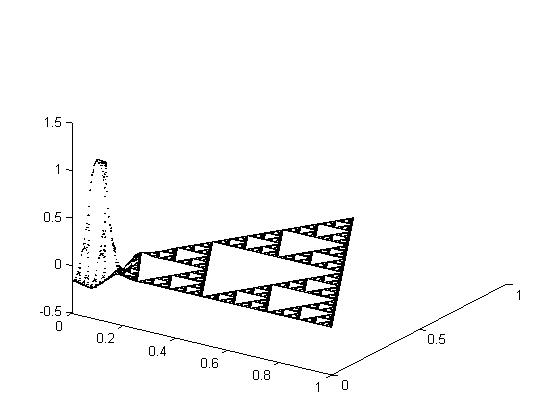}
    \end{subfigure}%
    ~ 
    \begin{subfigure}[b]{0.33\textwidth}
        \centering
        \includegraphics[width=\textwidth]{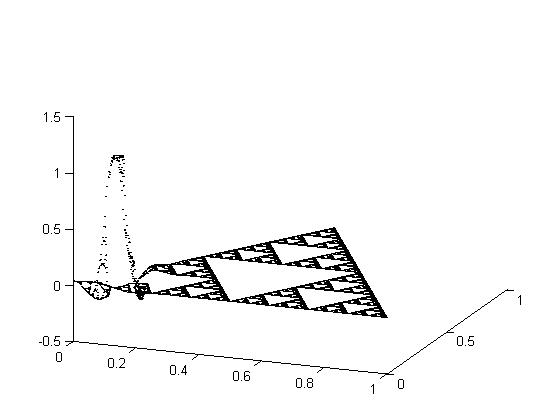}
    \end{subfigure}
    
    \begin{subfigure}[b]{0.33\textwidth}
        \centering
        \includegraphics[width=\textwidth]{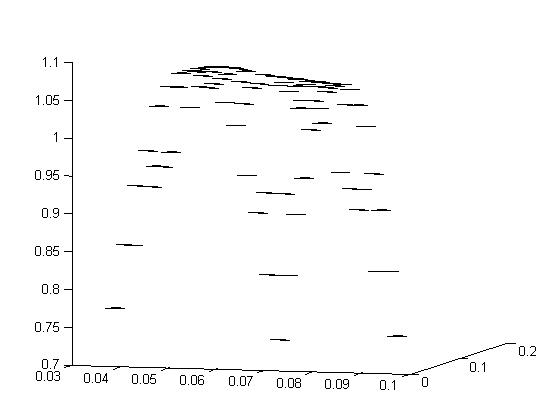}
    \end{subfigure}
    ~
    \begin{subfigure}[b]{0.33\textwidth}
        \centering
        \includegraphics[width=\textwidth]{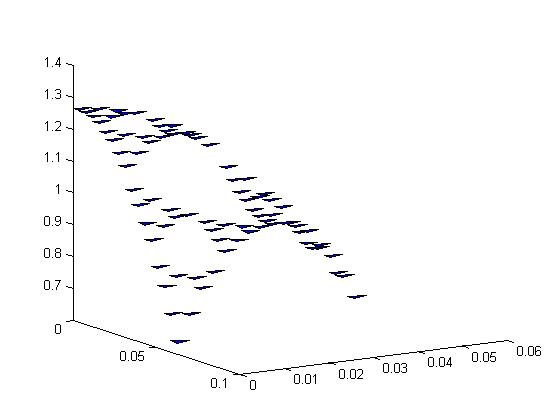}
        
    \end{subfigure}
    ~ 
    \begin{subfigure}[b]{0.33\textwidth}
        \centering
        \includegraphics[width=\textwidth]{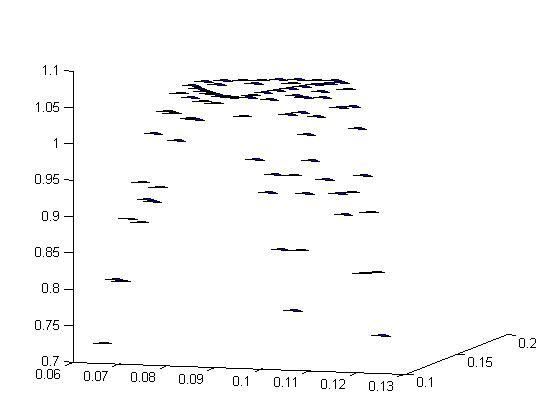}
    \end{subfigure}
    
    \caption{Some level 4 sampling functions (up to $D_{3}$ symmetries). The bottom row features zoomed in versions of the functions on the top row.}
\end{figure}

\begin{figure}
    \centering
    \begin{subfigure}[b]{0.33\textwidth}
        \centering
        \includegraphics[width=\textwidth]{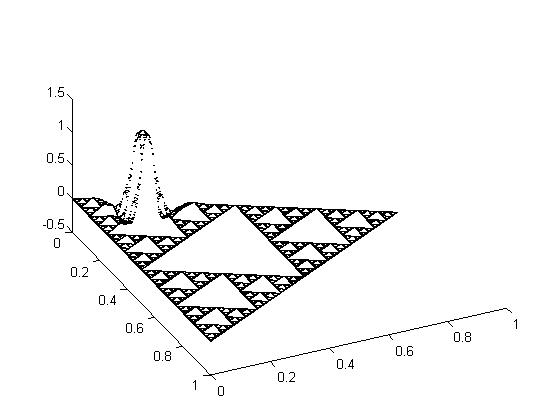}
        
    \end{subfigure}%
    ~ 
    \begin{subfigure}[b]{0.33\textwidth}
        \centering
        \includegraphics[width=\textwidth]{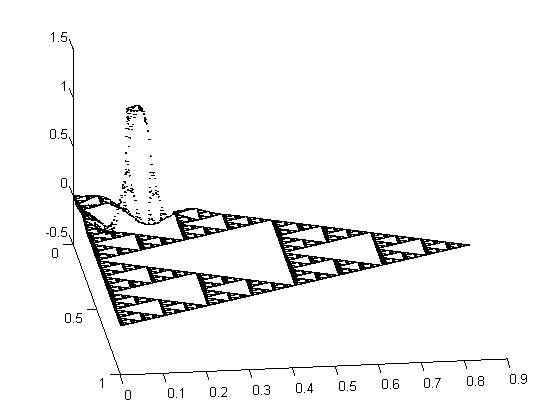}
    \end{subfigure}%
		~
    \begin{subfigure}[b]{0.33\textwidth}
        \centering
        \includegraphics[width=\textwidth]{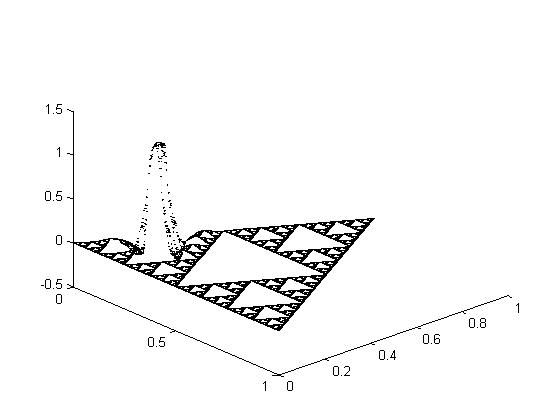}
    \end{subfigure}%

    \begin{subfigure}[b]{0.33\textwidth}
        \centering
        \includegraphics[width=\textwidth]{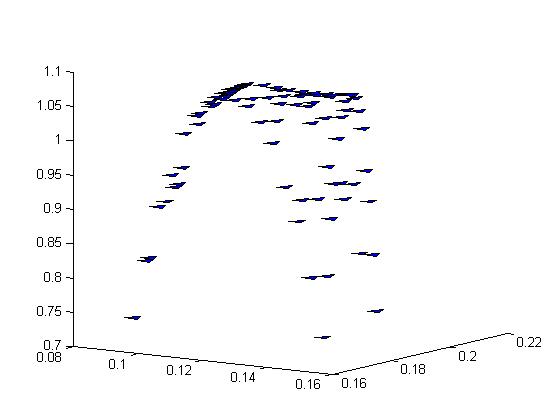} 
    \end{subfigure}
    ~
    \begin{subfigure}[b]{0.33\textwidth}
        \centering
        \includegraphics[width=\textwidth]{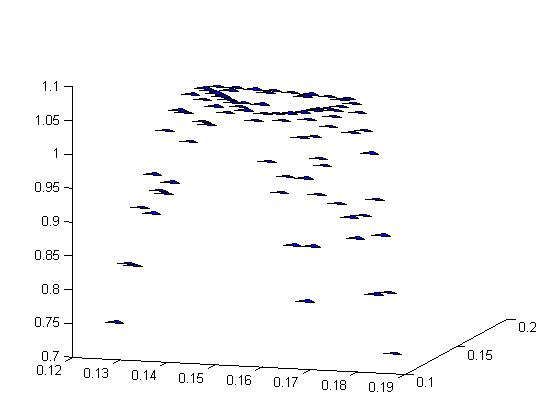}
    \end{subfigure}
    ~ 
    \begin{subfigure}[b]{0.33\textwidth}
        \centering
        \includegraphics[width=\textwidth]{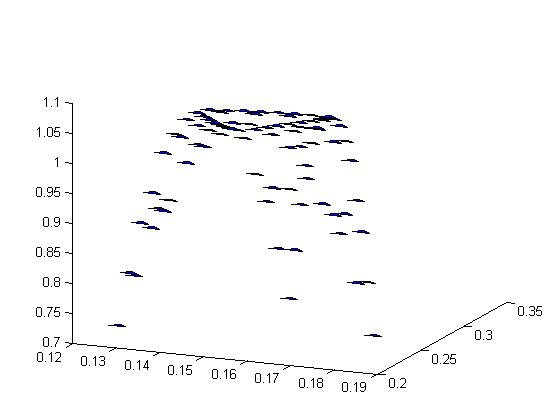}
    \end{subfigure}
    
    \caption{More level 4 sampling functions (up to $D_{3}$ symmetries). The bottom row features zoomed in versions of the functions on the top row.}
\end{figure}

\begin{figure}
    \centering
    \begin{subfigure}[b]{0.33\textwidth}
        \centering
        \includegraphics[width=\textwidth]{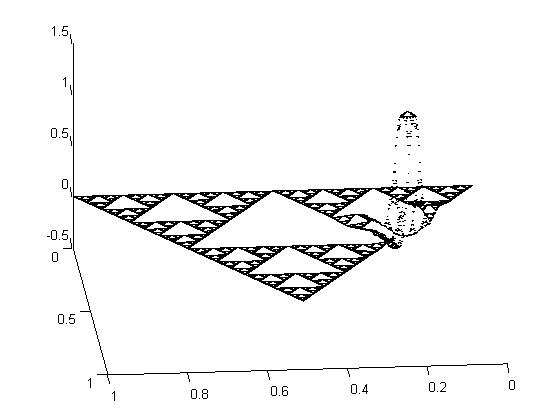}
    \end{subfigure}%
    ~ 
    \begin{subfigure}[b]{0.33\textwidth}
        \centering
        \includegraphics[width=\textwidth]{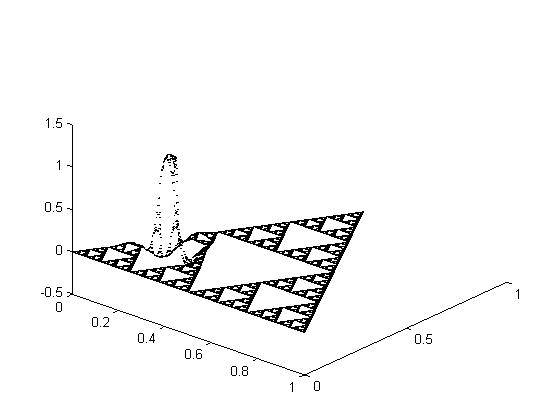}
    \end{subfigure}%
    ~
    \begin{subfigure}[b]{0.33\textwidth}
        \centering
        \includegraphics[width=\textwidth]{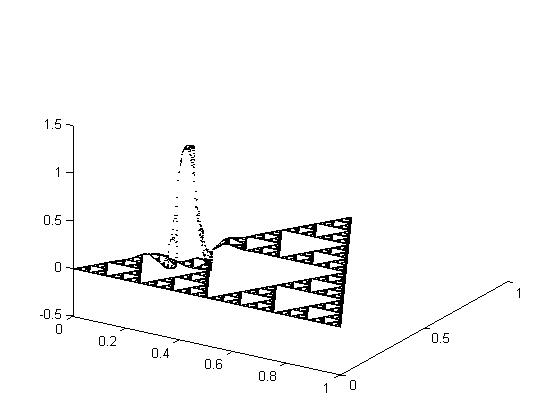}
    \end{subfigure}%
    
    \begin{subfigure}[b]{0.33\textwidth}
        \centering
        \includegraphics[width=\textwidth]{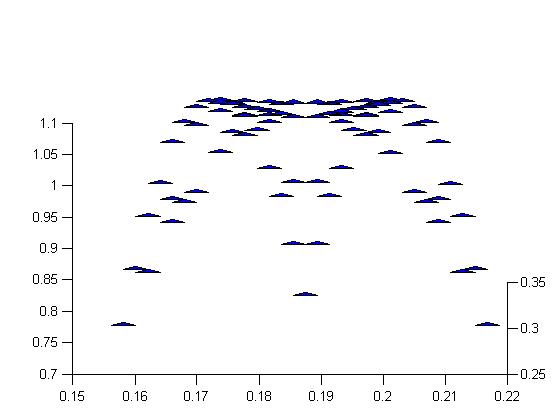}
    \end{subfigure}
    ~ 
    \begin{subfigure}[b]{0.33\textwidth}
        \centering
        \includegraphics[width=\textwidth]{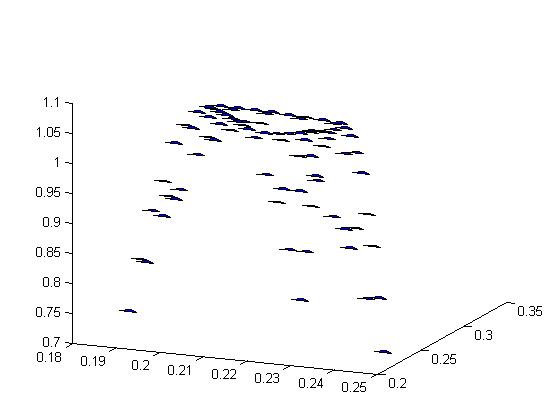}
    \end{subfigure}
    ~ 
    \begin{subfigure}[b]{0.33\textwidth}
        \centering
        \includegraphics[width=\textwidth]{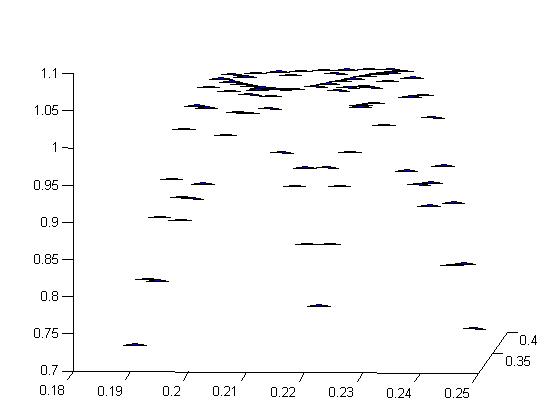}
    \end{subfigure}
    
    \caption{More level 4 sampling functions (up to $D_{3}$ symmetries). The bottom row features zoomed in versions of the functions on the top row.}
\end{figure}

\begin{figure}
    \centering
    \begin{subfigure}[b]{0.33\textwidth}
        \centering
        \includegraphics[width=\textwidth]{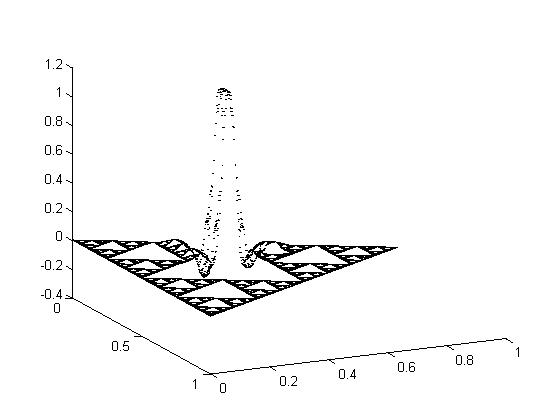}
        
    \end{subfigure}%
    ~
    \begin{subfigure}[b]{0.33\textwidth}
        \centering
        \includegraphics[width=\textwidth]{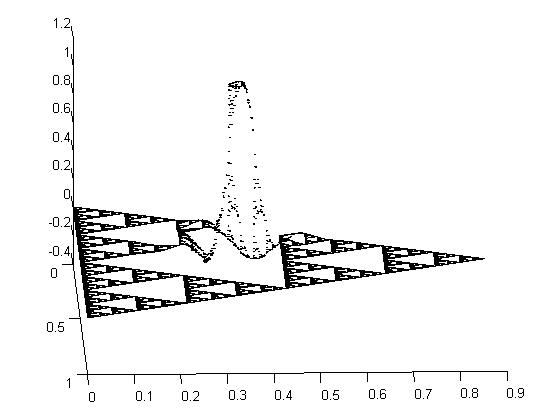}
    \end{subfigure}%
    ~ 
    \begin{subfigure}[b]{0.33\textwidth}
        \centering
        \includegraphics[width=\textwidth]{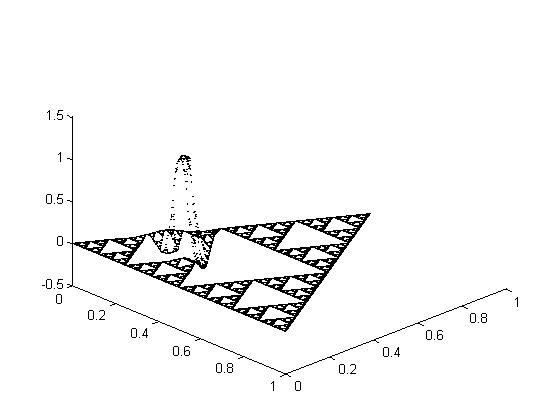}
    \end{subfigure}%
    
    \begin{subfigure}[b]{0.33\textwidth}
        \centering
        \includegraphics[width=\textwidth]{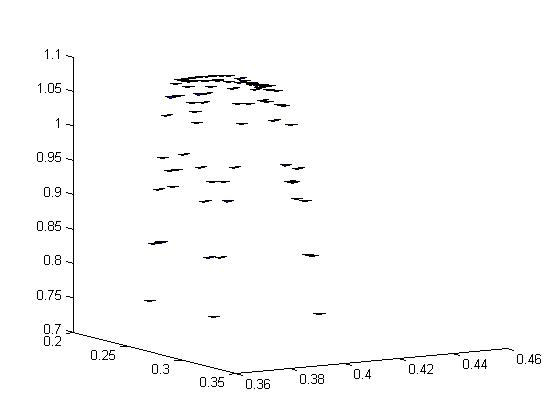}
        
    \end{subfigure}
    ~
    \begin{subfigure}[b]{0.33\textwidth}
        \centering
        \includegraphics[width=\textwidth]{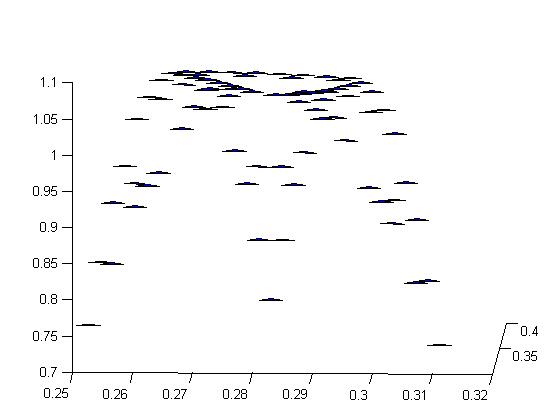}
    \end{subfigure}
    ~ 
    \begin{subfigure}[b]{0.33\textwidth}
        \centering
        \includegraphics[width=\textwidth]{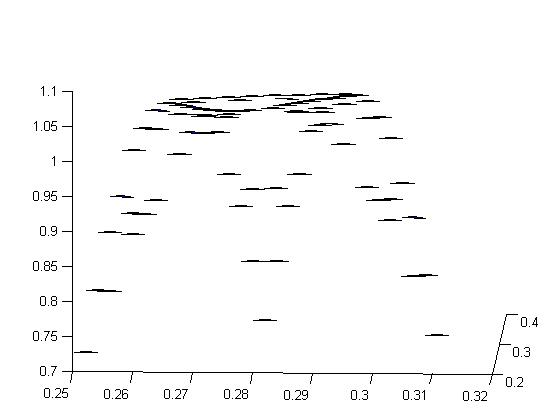}
    \end{subfigure}
    
    \caption{More level 4 sampling functions (up to $D_{3}$ symmetries). The bottom row features zoomed in versions of the functions on the top row.}
\end{figure}

\begin{figure}
    \centering
    \begin{subfigure}[b]{0.33\textwidth}
        \centering
        \includegraphics[width=\textwidth]{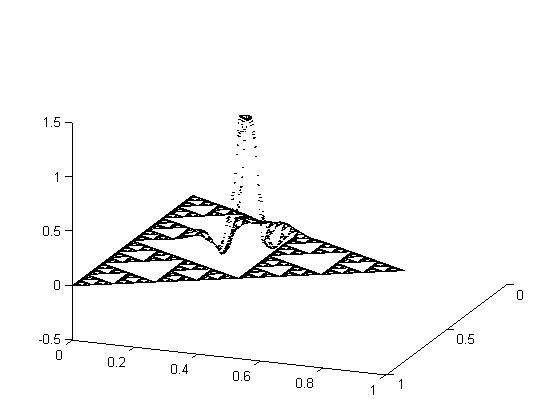}
        
    \end{subfigure}%
    ~ 
    \begin{subfigure}[b]{0.33\textwidth}
        \centering
        \includegraphics[width=\textwidth]{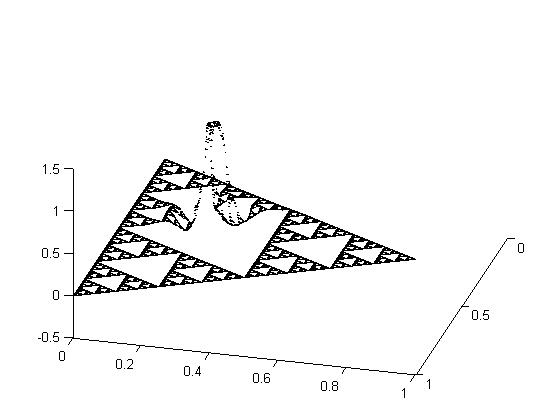}
    \end{subfigure}%

    \begin{subfigure}[b]{0.33\textwidth}
        \centering
        \includegraphics[width=\textwidth]{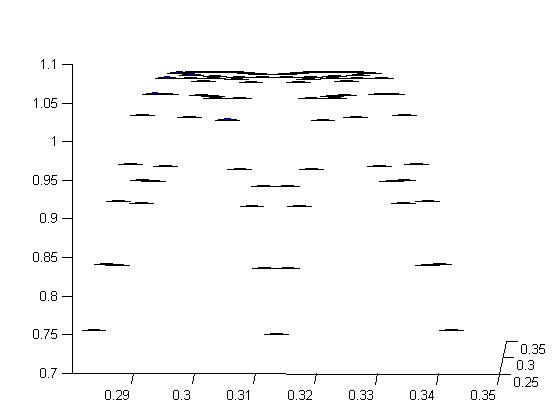}
        
    \end{subfigure}
		~
    \begin{subfigure}[b]{0.33\textwidth}
        \centering
        \includegraphics[width=\textwidth]{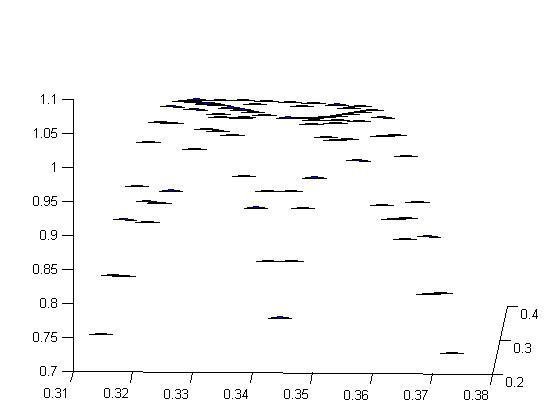}
    \end{subfigure}
    
    \caption{The remaining level 4 sampling functions (up to $D_{3}$ symmetries). The bottom row features zoomed in versions of the functions on the top row.}
\end{figure}
\subsection{Sampling on Infinite Blowups}

Let $\{i_{j}\}$ be an infinite sequence such that $i_{k} = 0,1,2 $, with the condition that two of the integers $0, 1,$ and $2$ occur infinitely often. Then, for $K = SG$, we define the infinite blow-up $K_{\infty}$ by

$$K_{\infty} = \bigcup_{m=1}^{\infty}F_{i_{1}}^{-1}...F_{i_{m}}^{-1}K.$$

\noindent and let $K_{(m)} = F_{i_{1}}^{-1}...F_{i_{m}}^{-1}K.$ Note that the sets $K_{(m)}$ are nested. Note also that for any $m \in \mathbb{Z},$ $K_{\infty}$ is an infinite union of essentially disjoint $m$-cells. The Laplacian easily extends from $K$ to $K_{\infty}.$ The spectral theory of the Laplacian on $K_{\infty}$ was described explicitly by Teplyaev \cite{teplyaev1998spectral}. See \cite{MR2246975} for a concise exposition.

The Laplacian on $K_{\infty}$ has a pure point spectrum consisting of the eigenvalues $5^{-m} \lambda_{k},$ where $\lambda_{k}$ is a Neumann eigenvalue on $K,$ each with infinite multiplicty, and each eigenspace has a basis of compactly supported eigenfunctions that are expressed as $u \circ F_{i_{1}}^{-1} \circ ... \circ F_{i_{m}}^{-1}$ extended by zero for $u$ an eigenfunction on $K$ satisfying both Dirichlet and Neumann condititions (with Dirichlet conditions being that $u = 0$ on the boundary points as usual). A function on $K_{\infty}$ is said to be bandlimited if it is a limit of linear combinations of eigenfunctions with eigenvalues at most $B.$ In particular, we want to take $B = b5^{m}$ to be the same bandlimit for level $m$ cells as in $K,$ so now we may also want to consider larger cells by allowing $m$ to be negative. We call this being $m$-bandlimited.

\begin{thm}
Assume Conjectures 5.1 and 5.2 are valid. Then for any cell $C$ of level $m$ there exists a sampling function $\psi_{C}$ that is $m$-bandlimited and $A_{C'}(\psi_{C}) = \delta_{C,C'}$ where $C'$ is any $m$-cell in $K_{\infty}.$ It follows that any $b5^{m}$-bandlimited function $f$ on $K_{\infty}$ is uniquely determined by its averages $A_{C}(f)$ over $m$-cells.
\end{thm}
\begin{proof}
Fix a cell in $K_{\infty}$. For simplicity, assume it is in $K$, say $F_{w}K = F_{w_{1}}...F_{w_{m}}K.$ Let $\psi_{0}$ be the sampling function of $F_{w}K$ in $K$ such that the average $A_{w'}(\psi_{0})$ of $\psi_{0}$ on $F_{w'}K$ equals $\delta_{ww'}$ for all $\lvert w' \rvert = m$.

Now $\psi_{0}$ has an extension to $K_{\infty}$ that is $m$-bandlimited and of compact support (modulo constant functions), since each non-constant Neumann eigenfunction on $K$ extends to a Neumann eigenfunction on $K_{\infty}$ that is supported in $K_{(n)}$ for some $n$ (this $n$ depends on the sequence $i_{1},i_{2},...$). Call the extension $\widetilde{\psi}_{0}.$ Note that not only is $\widetilde{\psi}_{0}$ bandlimited on the top, but it is also bandlimited on the bottom, because we do not add any eigenfunctions with eigenvalue below the smallest nonzero Neumann eigenvalue on $K$ except the constant term.

Note that the average of $\widetilde{\psi}_{0}$ on any $m$-cell $F_{w'}K$ of $K$ satisfies $\widetilde{\psi}_{0} = \delta_{ww'}$, but for other $m$-cells in $K_{\infty}$ we have no information.

Next we want to construct $\widetilde{\psi}_{1}$ on $K_{\infty}$ that is $m$-bandlimited and of compact support (modulo constants) such that $A_{w'} = \delta_{ww'}$ for every $m$-cell $F_{w'}K$ in $K$ and the average of $\widetilde{\psi}_{1}$ on any $m$-cell in $K_{(1)} \symbol{92} K.$ So we want to take the containment

$$F_{w}K \subseteq K \subseteq F_{i_{1}}^{-1}K$$

\noindent and apply $F_{i_{1}}$ to it:

$$F_{i_{1}}F_{w}K \subseteq K.$$

So $F_{i_{1}}F_{w}K$ is an $(m+1)$-cell in $K$, so it has an $(m+1)$-bandlimited sampling function $\varphi_{1}$, and $A_{w'}(\varphi_{1}) = \delta_{(i_{1}w)w'}.$ Let $\psi_{1} = \varphi_{1} \circ F_{i_{1}}$ defined on $K_{1} = F_{i_{1}}^{-1}K.$ Then $\psi_{1}$ has the desired averages on $m$-cells in $K_{(1)}$. Also, $\psi_{1}$ is $m$-bandlimited so there is an extension $\widetilde{\psi}_{1}$ to $K_{\infty}$ that is bandlimited and compactly supported modulo constants.

We can easily compute the constant contributions to $\widetilde{\psi}_{0}$ and $\widetilde{\psi}_{1}$ since all nonconstant eigenfunctions have total integral zero and $\int_{K} \psi_{0} = \frac{1}{3^{m}} = \int_{K_{(1)}} \psi_{1},$ so the constant is $\frac{1}{3^{m}}$ for $\widetilde{\psi}_{0}$ and $\frac{1}{3^{m+1}}$ for $\widetilde{\psi}_{1}.$

Iterating this argument, we obtain a sequence $\widetilde{\psi}_{0}$, $\widetilde{\psi}_{1}$, $\widetilde{\psi}_{2}$,... of $m$-bandlimited functions such that the $\widetilde{\psi}_{j}$ are of compact support on $K_{\infty}$ such that $A_{w}(\widetilde{\psi}_{j}) = 1$ and the average on any other $m$-cell in $K_{(j)}$ is 0

Now, assume that the two conjectures of the previous section are true. Conjecture 5.1 says that the sampling functions on $K$ are uniformly bounded, which is clearly equivalent to the $\widetilde{\psi}_{j}$ being uniformly bounded on each $K_{(j)}$. It is also easy to see that Conjecture 5.2 is equivalent to $\int_{K_{(j)}} \left| \widetilde{\psi}_{j} \right|^{2} \le c3^{-m}$ for all $j$.

Now, fix $K_{(n)}$. For $j \ge n$, the restriction of $\widetilde{\psi}_{j}$ to $K_{(j)}$ is $m$-bandlimited, so

$$\int_{K_{(j)}} \left|\Delta\widetilde{\psi}_{j}\right|^{2} \le M_{m}^{2}\int_{K_{(j)}} \left|\widetilde{\psi}_{j}\right|^{2}$$

\noindent where $M_{m}$ is the highest frequency in the $m$-band. Combining this with the obvious estimate

$$\int_{K_{(n)}} \left|\Delta\widetilde{\psi}_{j}\right|^{2} \le \int_{K_{(j)}} \left|\Delta\widetilde{\psi}_{j}\right|^{2}$$

\noindent and the aforementioned consequences of Conjecture 5.2, we have the estimate

$$\int_{K_{(n)}} \left|\Delta\widetilde{\psi}_{j}\right|^{2} \le C_{n}$$

\noindent for all $j \ge n.$ We can use this to get the H\"{o}lder estimate

$$\left|\widetilde{\psi}_{j}(x) - \widetilde{\psi}_{j}(y)\right| \le cR(x,y)^{\beta} \left( \int_{K_{(n)}} \left|\Delta\widetilde{\psi}_{j}\right|^{2} \right)^{\frac{1}{2}}$$

\noindent for all $x,y \in K_{(n)}$ for some $\beta$, where $R$ is the resistance metric. This gives us uniform equicontinuity on $K_{(n)}$.

As the sequence $\{\widetilde{\psi}_{j}\}$ is uniformly bounded and uniformly equicontinuous, we can apply Arzela-Ascoli to find a subsequence $\{\widetilde{\psi}_{j_{k}}\}$ converging uniformly on each $K_{(n)}$. Let $\psi = \lim_{k \to \infty} \widetilde{\psi}_{j_{k}}.$ Then $\psi$ is $m$-bandlimited, since if $\varphi$ is any high frequency eigenfunction, $\int \psi \varphi = \lim_{k \to \infty} \int \widetilde{\psi}_{j_{k}} \varphi = 0$. Also, the average of $\psi$ on $F_{w}K$ is 1, while the average of $\psi$ on any other $m$-cell is 0, so $\psi$ is a sampling function for $F_{w}K$ on $K_{\infty}$.
\end{proof}

\section{The Case of $SG_{3}$}

We consider the unit equilateral triangle $T$ in $\mathbb{R}^{2}$ with vertices $q_{0}, q_{1},$ and $q_{2}$ as in Section 2. We subdivide $T$ into 9 equilateral triangles of side length $\frac{1}{3},$ six of which are upward-pointing whereas the other three are downward pointing. See Figure 25.

\begin{figure}
\begin{center}
\includegraphics[scale=.6]{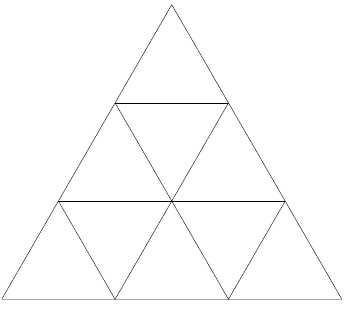}
\end{center}
\caption{$\zeta_{1}$, the first level approximation of $SG_{3}$}
\end{figure}

\noindent Let $T_{1},..., T_{6}$ be the six upward-pointing triangles of side length $\frac{1}{3}$ in Figure 24 (order doesn't matter). We let $G_{i}: \mathbb{R}^{2} \to \mathbb{R}^{2}$ be the homeomorphism that maps $T$ to $T_{i}.$ Analogous to $SG,$ we have:

\begin{defn}
$SG_{3}$ is the unique non-empty compact subset of $\mathbb{R}^{2}$ satisfying

$$SG_{3} = \bigcup_{i=0}^{5} G_{i}(SG_{3}) \eqno(6.1)$$
\end{defn}

\noindent The integral, energy, and Laplacian are defined similarly to those for $SG$. See \cite{drenning2009spectral} for complete details.

\begin{figure}[!h]
    \centering
    \begin{subfigure}[b]{0.5\textwidth}
        \centering
        \includegraphics[width=\textwidth]{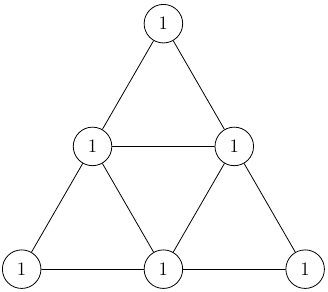}
        
    \end{subfigure}%
    ~ 
    \begin{subfigure}[b]{0.5\textwidth}
        \centering
        \includegraphics[width=\textwidth]{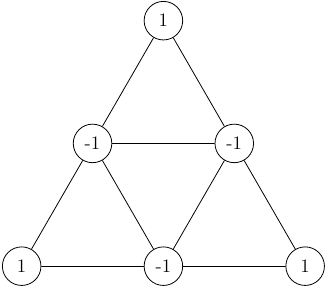}
        
    \end{subfigure}
    
    \begin{subfigure}[b]{0.5\textwidth}
        \centering
        \includegraphics[width=\textwidth]{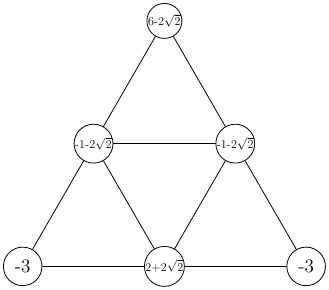}
    \end{subfigure}%
    ~ 
    \begin{subfigure}[b]{0.5\textwidth}
        \centering
        \includegraphics[width=\textwidth]{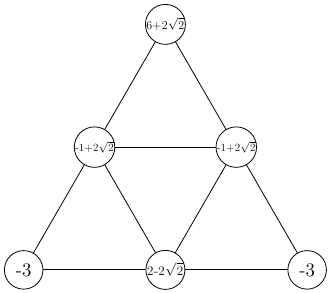}
    \end{subfigure}
    
    \begin{subfigure}[b]{0.5\textwidth}
        \centering
        \includegraphics[width=\textwidth]{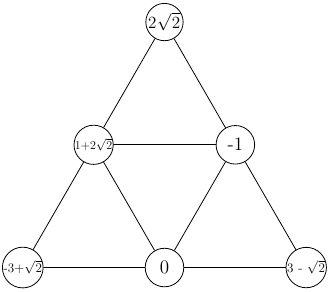}
    \end{subfigure}%
    ~ 
    \begin{subfigure}[b]{0.5\textwidth}
        \centering
        \includegraphics[width=\textwidth]{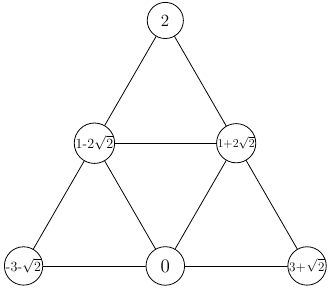}
    \end{subfigure}
    \caption{The average value functions on $\xi_{1}$ of the first six elements of the Neumann eigenbasis of $\zeta_{1}$ }
\end{figure}

The vertex graph approximations of $SG_{3}$, which we denote $\zeta_{m},$ are defined analogously to the $\beta_{m}$ for $SG.$ For $SG_{3},$ we let words of length $n$ be elements of $\mathbb{Z}_{6}^{n},$ and we keep the same convention that for a word $w = (w_{1},...,w_{n}) \in \mathbb{Z}_{6}^{n},$ $G_{w} = G_{w_{n}} \circ...\circ G_{w_{1}}.$ We let $\zeta_{0}$ be the graph associated with $T.$ For $m \geq 1,$ let

$$V_{m} = \bigcup_{i=0}^{5} G_{i}(V_{m-1}),$$

\noindent and let $\zeta_{m}$ be the graph with vertices $V_{m}$ such that $x, y \in V_{m}$ are connected by exactly one edge if and only if $x \neq y$ and $x,y \in F_{w}(T)$ for some word $w \in \mathbb{Z}_{6}^{m}$ of length $m.$ For example, $\zeta_{1}$ is given by the graph corresponding to Figure 24. If we let $V_{\ast} = \lim_{m \to \infty} V_{m},$ then $SG_{3}$ is the completion of $V_{\ast}$ in $\mathbb{R}^{2}.$

We say that the set $K$ is an {\bf{$m$-cell of $SG_{3}$}} if $K = G_{w}(SG_{3}).$ We say the graph $W$ is an $m$-cell on $\zeta_{n}$ if $W = G_{w}(\zeta_{n-m})$ for some $w \in \mathbb{Z}_{6}^{m}.$ The boundary vertices of an $m$-cell are the images of $q_{0}, q_{1},$ and $q_{2}$ under the map $G_{w}$ corresponding to the $m$-cell.

We define the average cell graphs $\xi_{m}$ by letting every vertex of $\xi_{m}$ correspond to an $m$-cell of $\zeta_{m},$ and connecting two vertices by an edge if and only if the corresponding $m$-cells share exactly one vertex in common. We see that $\xi_{0}$ is just the graph with one vertex and zero edges, and that $\xi_{1} = \beta_{1}.$ Similar to the situation with $\Gamma_{m},$ every vertex on $\xi_{m}$ splits into six different vertices on $\xi_{m+1}.$ We can define averages analogously to (2.4) and (2.5), and we force functions on $\xi_{m}$ to satisfy the immediate analogues of (3.1) and (3.2). Namely, ssume that a vertex $x$ on $\xi_{m}$ splits into six distinct vertices $w_{1},...,w_{6}$ on $\xi_{m+1}.$ Then we say that a function $g$ on $\xi_{m+1}$ is a {\bf{continuation of a function $f$ on $\xi_{m}$}} provided that

$$f(x) = \frac{1}{6} \sum_{i=1}^{6} g(w_{i}) \eqno(6.1)$$

In \cite{drenning2009spectral}, it was shown that $SG_{3}$ and the $\zeta_{m}$ possess the spectral decimation property. It is then reasonable to ask whether the $\xi_{m}$, with the usual graph Laplacian, also have the spectral decimation property. We can answer that in the affirmative.

\begin{thm}
Let $u$ be an eigenfunction on $\xi_{m}$ with eigenvalue $\lambda_{m}$. Then, $u$ can be continued to at most two eigenfunctions on $\xi_{m+1}$ with eigenvalues $\lambda^{(1)}_{m+1}$ and $\lambda^{(2)}_{m+1}$. Furthermore, for each $\lambda^{(k)}_{m+1}$, the corresponding continuation is unique. Conversely, if $u$ is an eigenfunction on $\xi_{m+1}$ with eigenvalue $\lambda^{(1)}_{m+1}$ or $\lambda^{(2)}_{m+1}$, then $u'$, the function on $\xi_{m}$ defined by

$$u'(x) = \frac{1}{6} \sum_{i=0}^{5}u(G_{i}x)$$

\noindent is an eigenfunction on $\xi_{m}$ with eigenvalue $\lambda_{m}.$ The relationship between $\lambda_{m}$ and $\lambda_{m+1}^{(k)}$ is given by

$$\lambda_{m} = \frac{3(\lambda_{m+1}-5)(\lambda_{m+1}-4)(\lambda_{m+1}-3)\lambda_{m+1}}{3\lambda_{m+1}-14} \eqno(6.2)$$
\end{thm}

\noindent The proof is analogous to that of Theorem 3.2, and will be skipped. For our analysis, the formulas for the continuation do not matter nearly as much as the relationship between the eigenvalues. The eigenvalue relation satisfied by the $\zeta_{m}$ graphs is 

$$\lambda_{m} = \frac{(\lambda_{m+1}^{2}-9\lambda_{m+1}+19)(\lambda_{m+1}-4)\lambda_{m+1}}{\lambda_{m+1}-6} \eqno(6.3)$$

\noindent So the eigenvalue relations for $\zeta_{m}$ and $\xi_{m}$ are not equal to one another. This is in stark contrast to the $SG$ case, where the relations were the same. Recall that the fact that the eigenvalue relations for $\beta_{m}$ and $\Gamma_{m}$ were the same was crucial in the proof of Lemma 4.3, which was itself needed to prove Theorem 4.2. The fact that this does not repeat for $SG_{3}$ lends one to believe that an analogue of Theorem 4.2 in this case may not exist.

In fact, we can show that the average values of the first six Neumann eigenfunctions on $\zeta_{1}$ (with first six again meaning the six smallest eigenvalues of the basis) are not all eigenfunctions on $\xi_{1}$ for any Laplacian on $\xi_{1}.$ See Figure 26 for the average values of these functions (the actual functions themselves can be found in \cite{drenning2009spectral}, but we do not need them). One can attempt to compute a Laplacian matrix for which these six functions are eigenfunctions on $\xi_{1}.$ However, by diagonalizing the Laplacian matrix, we find that all eigenvalues are necessarily equal, hence the Laplacian matrix is a multiple of the identity matrix, which is absurd.

\bibliographystyle{abbrv} 
\bibliography{sg}

\end{document}